%% file: main.tex
\documentclass[letterpaper,11pt]{amsart} 

\usepackage[portrait,margin=3cm]{geometry} 
\include{style}

\begin{document}
\title[Random optimization problems at fixed temperatures]{Random optimization problems at fixed temperatures}
\author[Dey]{Partha S.~Dey}
\address{\normalfont (P.S.D.) 
Department of Mathematics,
University of Illinois Urbana Champaign, 
1409 W Green Street, Urbana, IL, USA 61801}
\email{psdey@illinois.edu}
\author[Terlov]{Grigory Terlov}
\address{\normalfont (G.T.) 
Department of Statistics and Operations Research,
University of North Carolina at Chapel Hill, NC, USA 27514}
\email{gterlov@unc.edu}

\begin{abstract} 
This article considers a class of disordered mean-field combinatorial optimization problems. We focus on the Gibbs measure, where the inverse temperature does not vary with the size of the graph and the edge weights are sampled from a general distribution under mild assumptions. 
Our results consist of the Law of Large Numbers and Central Limit Theorems for the log-partition function, the weight of a typical configuration, and the Gibbs average in both quenched and annealed forms. 
We also derive quenched Poisson convergence for the size of the intersection of two independent samples, yielding replica symmetry of the model.  Applications cover popular models from the literature, such as the Minimal Matching Problem, Traveling Salesman Problem, and Minimal Spanning Tree Problem, on a sequence of deterministic and random dense graphs of increasing size.
\end{abstract}
\subjclass[2020]{Primary: 60F05, 82B44, 90C27}
\keywords{Random Assignment, Minimal Matching, Traveling Salesman Problem, Minimal Spanning Tree, Central limit theorem, Partition function, Cluster expansion}

\maketitle
\setcounter{tocdepth}{1}
\tableofcontents

\setcounter{tocdepth}{1}
\tableofcontents

\section{Introduction}

Given an undirected graph with weights assigned to its edges, combinatorial optimization problems ask to find a subgraph from a specific family that optimizes the total weight on its edges. When the underlying weights are random, such problems are called random (combinatorial) optimization problems. Some of the most popular problems in the literature are a minimal matching problem (MMP), traveling salesman problem (TSP), minimal spanning tree problem (MSTP), and minimal $k$-factor problem, see Section~\ref{sec: Review}. Although these problems are often easy to state, their analysis is complicated, and most of the progress so far has required many new ideas from different fields of study. Thus, throughout the years, this topic facilitated and benefited from many innovations in a variety of fields such as physics, discrete probability, computer science, engineering, and algorithm design~\cite{donath1969algorithm,mezard1985replicas, papadimitriou1998combinatorial, SteeleBook,YukichBook}.

In the mean-field case, \ie~when the underlying graph is dense, and the edge-weights are i.i.d., the connection between random optimization problems and statistical mechanics was established in the early 1980s~\cite{kirkpatrick1981lecture} and was particularly influential in shaping the field. In the mid-1980s physicists M{\'e}zard, Krauth, and Nobel laureate Parisi  ~\cite{mezard1985replicas,mezard1987spin,parisi1992spin,krauth1989cavity,mezard1986mean} successfully used cavity and replica methods to get remarkably detailed predictions for MMP and TSP. One of their running assumptions was that these optimization problems are replica-symmetric at low temperatures, which does not hold for classical models in statistical mechanics such as spin glasses. Intuitively, replica symmetry means that two subgraph samples drawn from the Gibbs measure essentially share the same constant fraction of edges among them. 

The question of the validity of this assumption was later studied both numerically~\cite{Bacci, boettcher2000nature, brunetti1991extensive,cerf1997random,Houdayer, percus1999stochastic, sourlas1986statistical, percus1997voyageur} and from a physical theoretic point of view~\cite{mezard1987spin, pagnani2003near, parisi2001neighborhood}. Mathematically replica symmetry ansatz was verified only in one case by W\"astlund in 2012~\cite{WastlundReplicaRA}, who showed it for the MMP on the complete graph $K_{2n}$ with a particular choice of distribution of the weights. W\"astlund states that a similar analysis should hold for the MMP on the complete bipartite graph and for the TSP on the complete graph, however to our knowledge, it has not been confirmed.

To enable the statistical mechanics approach, one considers an auxiliary randomness on the space of all subgraphs, over which one is optimizing, on top of the random environment created by the random weights. Namely, let $\cS$ be the set of all subgraphs of $G$ that are of interest in the optimization problem. For $\pi\in \cS$, we denote its weight by $W(\pi)$. We now introduce the Gibbs measure
\begin{equation*}
\pr_{\gb}(\pi)=\frac{1}{Z(\gb)}\cdot \frac1{\abs{\cS}}\exp\left(-\gb W(\pi)\right),    
\end{equation*}
where $\gb$ is the \textit{inverse temperature} parameter  and $Z(\gb)$ is the normalizing constant known as the \textit{partition function}. Since most of the literature is concerned with identifying the optimal configuration $\pi$, known as the ground state, the inverse temperature tends to infinity as a function of the size of the graph. For instance, if $G$ is the complete graph on $N$ vertices with $\Exp(1)$ weights on its edges one considers $\gb=\wh{\gb}\cdot N$ for some constant $\wh{\gb}\in (0,\infty)$ (\eg~ see~\cite{kirkpatrick1981lecture} and~\cite[Chapter 7]{talagrand2003spin}). 

In this work, we consider \textit{fixed} $\gb\in\bR$. In this case, the resulting measure does not converge to the delta mass function on the optimal configuration but rather is biased towards configurations with smaller weights. One can interpret this case as a high-temperature regime with a general edge-weight distribution. In general, positive temperature versions of these problems are natural to consider as they often lead to interesting behavior. For example, the classical anti-ferromagnetic Ising model is a positive temperature version of the maximum cut problem. In such settings, many of the difficult aspects of optimization problems remain in place, such as the occurrence of frustrations and complicated structure of dependencies. However, it is also, in a sense, close to a function of the family of independent random variables, which, in our case, is made precise by techniques known as cluster expansion. We now briefly summarize our main contributions.

\begin{enumerate}
    \item Our setup is very general in the sense that we impose a very mild condition on the graph (Assumption~\ref{as:Graph}), weights (Assumption~\ref{as:weights}), and we consider a large class of optimization problems (Assumption~\ref{as:optimization}) that includes MMP, TSP, and MSTP. Although the bodies of work for these problems share many common themes, one can rarely derive abstract results that do not rely on particularities of the distribution of the weights and specific properties of a given optimization problem.
    \item We derive the law of large numbers (LLN) and central limit theorem (CLT) for 
    \begin{enumerate}
        \item the log-partition function (Theorem~\ref{thm:log part clt} and Corollary~\ref{thm:Limiting free energy}),
        \item the weight of a typical configuration (Theorem~\ref{thm:typical CLT} and Corollary~\ref{thm:typical LLN}),
        \item the Gibbs average (Theorem~\ref{thm:gibbs CLT}).
    \end{enumerate}
    \item We derive the Poisson limit theorem for the size of the intersection of two independently sampled configurations (Theorem~\ref{thm:poi limit}). This yields that the overlap is strongly concentrated at $0$; thus, the model is replica symmetric in a strong sense (Remark~\ref{rem:repsym}).
    \item We provide variations of our results tailored to multipartite and random graphs (Theorems~\ref{thm:log part cltp} and Observation~\ref{obs:random-graphs}, resp.). It is plausible that random optimization problems on these graphs satisfy our assumptions, and thus, our main results would apply directly. However, it turns out that verifying them on graphs beyond complete and complete bipartite graphs can be involved. We provide extensions of our results to illustrate how one can bypass some of those challenges.
\end{enumerate}

\subsection{Road map}
This paper is organized as follows. In Section~\ref{sec: Review}, we present some popular random optimization problems and review existing literature. After that, in Section~\ref{intro highlight}, we highlight the consequence of our results to one of the models. We follow up with the assumptions in Section~\ref{intro setup} and heuristics in Section~\ref{intro heuristic}. In Section~\ref{intro cluster} we cover the background on the cluster expansion - our primary technique throughout the paper. We provide the glossary of notations in Section~\ref{intro notations}. Section~\ref{sec mainres} is dedicated to the statement of our main results, while Section~\ref{sec:extensions of main res} covers their extensions to multipartite and random graphs. We apply our main results to various optimization problems in Section~\ref{sec appl}. Finally, we present the proofs of the main results in Section~\ref{sec proofs} and finish by discussing the results and further questions in Section~\ref{sec future}.

\subsection{Review of random optimization problems}\label{sec: Review}
In this section, we introduce various random optimization problems and review some of their history and relevant results. 
Due to the large volume of literature on this topic, we mostly focus on highlighting results in the mean-field case that are relevant to our setting. The list of references presented below is far from being complete and we refer the reader to the references therein for more details. Beyond the mean-field case, we refer the reader to the books~\cite{SteeleBook, YukichBook}.

A common theme throughout this section is that mathematical treatment of combinatorial optimization problems started in the literature in the first half of $20^{\textnormal{th}}$ century from the perspectives of graph theory and algorithmic design. Motivated by this progress, the randomized versions of these problems became of interest in the second half of $20^{\textnormal{th}}$ century and were often approached with techniques from statistical physics. Since then, the behavior of the average has been fairly understood in the majority of the settings. On the other hand, central limit theorems (CLT), although conjectured, were established only in a handful of situations (see~\cite[Section 5]{SouravSurvey}).

\subsubsection{Minimal Matching Problem (MMP)}\label{intro MMP}

Let $G=(V, E)$ be a graph, such that $\abs{V}$ is even, and sequence of weights on its edges $\{\go_e\}_{e\in E}$. The minimal matching problem, also known as the Assignment Problem, asks to find a perfect matching of vertices $(i,\pi(i))$ such that the sum of edge-weights between matched vertices is minimized, here $i\in V$ and $\pi(i)$ is a bijection between two equal size partitions of $V$. Most commonly, $G$ is taken to be the complete graph $K_{2n}$ or the complete bipartite graph $K_{n,n}$, when the weights $\{\go_e\}_{e\in E}$ are random they are usually taken to be independent and with continuous distributions so that the minimizer is unique.

As with many combinatorial optimization problems, MMP has no unique point of origin. The first mathematical treatment in the Euclidean case dates back to Monge~\cite{monge1781memoire}, who was studying the optimal transport of the earth. The question of finding a perfect matching in a graph dates back at least to 1916, and the work of K\"onig on bipartite dense graphs~\cite{konig1916graphen} that can be seen as a deterministic MMP where weights $\go_e$ take values in $\{0,1\}$. Kurtzberg first considered the randomized version of this question in 1962~\cite{kurtzberg1962approximation}. He considered the case of $\textrm{Uniform}(0,1)$ weights on the complete bipartite graph. Both deterministic and randomized variations attracted the attention of computer scientists in the second half of $20^{\textnormal{th}}$ century and led to various exciting developments in algorithm design~\cite{munkres1957algorithms, kurtzberg1962approximation, donath1969algorithm}. Partly motivated by this progress, it became of great interest to analyze and compute the asymptotic average weight of minimal matching~\cite{Walkup79, linearprograms, KarpUppbnd, Lazarus, CoppersmithSorkin}. 

In the 1980s this problem was connected to statistical physics in the sequence of works of M\'ezard and Parisi~\cite{mezard1985replicas,mezard1987spin,parisi1992spin,mezard1986mean}. These authors developed the replica method and applied it to the minimal matching problem on the complete bipartite graph with nonnegative weights whose distribution decays polynomially near zero. Their results predicted that the total weight of the minimal matching converges to $\zeta(2)$ as the graph size tends to infinity. Although the authors did not specify the mode of convergence, at least convergence in probability was implied. 
This was later established rigorously in the case of exponentially distributed weights and the convergence in expectation in two independent works, one by Aldous~\cite{AldousZeta} and another by Nair, Prabhakar, and Sharma~\cite{NPS}. Generalizations of these results to the $k$ matching, where $G=K_{n,kn}$ and each of the $n$ vertices is matched with a set of $k$ vertices, were shown in~\cite{NPS, LinussonWastlund}. Convergence in probability of the weight of the minimal matching follows from works of Aldous~\cite{AldousAsymptotics,AldousZeta}.  Analogous results were shown for the complete graph with the even number of vertices~\cite{Wastlundcomplete,hessler2008concentration}. In 2005, W\"astlund~\cite{wastlund2005variance} gave an independent proof for the convergence of the average to $\zeta(2)$ and obtained formulas for the variance and the higher moments. He showed that the variance is $$\frac{4\zeta(2)-4\zeta(3)}{n}+O\left(\frac{1}{n^2}\right),$$ 
which slightly differs in the constant from the value of $2$ originally conjectured by Alm and Sorkin~\cite{alm2002exact}.  As mentioned above, the works of M\'ezard and Parisi assumed replica symmetry of the model shown only in the case of the MMP on $K_{2n}$~\cite{WastlundReplicaRA}. 
Talagrand \cite[Chapter 7]{talagrand2003spin} considered the MMP on $K_{n,m}$, where $m=\lfloor n(1+\ga)\rfloor$ for some $\ga\in (0,\infty)$. He showed explicit limiting behavior of the partition function for all $\gb\le \gb(\ga)$. Finally, the conjecture of asymptotic normality of the weight of the minimal matching remains open~\cite[Conjecture 1.1]{wastlund2005variance} and~\cite[Section 11]{hessler2008concentration}. Central limit theorems for MMP  were proven recently in 2019 by Barrio and Loubes in the Euclidean case~\cite{MMPCLT} and by Cao in 2021 in the case of sparse Erd\H{o}s-R\'enyi graphs with exponential weights on the edges~\cite{CaoCLT}. We briefly highlight the applications of our results to the MMP on $K_{n,n}$ in Section~\ref{intro highlight}. For the rest of the applications to this model, see Section~\ref{sec: MMP}.

\subsubsection{Traveling Salesman Problem (TSP)}\label{intro TSP}
The traveling salesman problem asks to find a Hamiltonian cycle, that is a cycle that goes through every vertex of the graph exactly once, with the smallest combined weights of its edges.

The first mathematical treatment of the traveling salesman problem dates back to $19^{\textnormal{th}}$ century and the works of Hamilton and Kirkman (see~\cite{biggs1986graph}). The first general formulation of the problem is due to Menger in 1928~\cite{menger1928theorem}, followed by a spike of research on this problem in various forms. For more on the history of this problem, we refer to~\cite{schrijver2005history}. Similarly to MMP, in the 1980s TSP a randomized version of TSP became of interest and was heavily studied from the perspective of statistical physics~\cite{mezard1986replicaTSP, krauth1989cavity, percus1999stochastic, sourlas1986statistical,vannimenus1984statistical, mezard1985replicas} using similar techniques such as cavity and replica methods as well as numerical simulations. In the line of work~\cite{karp1985probabilistic, FriezeSorkinTSP, DyerFriezeTSP, KarpTSP} using patching algorithms, it was rigorously shown that the weight of the minimum Hamiltonian cycle on the directed $K_n$ with independent Uniform[0,1] random weights with high probability is the same as the weight of the minimal matching. In 2004~\cite{FriezeTSP}, Frieze showed a similar result in the symmetric version of this problem, \ie~on undirected $K_n$, this time proving that with high probability the weight of the optimal Hamiltonian cycle is equal to the weight of the optimal 2-factor (a 2-regular subgraph or, equivalently, a union of cycles). W\"astlund computed the limiting expected value of the minimum Hamiltonian cycle on complete bipartite graph~\cite{wastlund2006limit} and on complete graph~\cite{wastlund2006traveling, WastlundTSP}. Although the asymptotic normality of the minimum weight over Hamiltonian cycles was asked in both Euclidean and mean-field settings, see~\cite[Section 5]{SouravSurvey}, we are unaware of any results in this direction beyond concentration inequalities derived in \cite{RheeTalagrandTSP}.
For applications of our result to this model, see Section~\ref{sec: TSP}.

\subsubsection{Minimal Spanning Tree Problem (MSTP)}\label{intro MST}

The minimal spanning tree (MST) problem asks to find a spanning tree with the minimal combined weight of its edges.

Although this problem likely appeared independently in several sources, as we know, it was first formulated by Bruv\r{u}ka in 1926~\cite{boruuvka1926jistem}. It was then popularized in the 1950s by works of Kruskal~\cite{kruskal1956shortest} and Prim~\cite{prim1957shortest}. For more on the history of the origins of this problem and its algorithmic solutions, we refer to~\cite{graham1985history}. One aspect in which this problem is different from the minimal matching and the minimal Hamiltonian cycle is that MST is a matroid, and hence, it can be obtained by a greedy algorithm in polynomial time. This, in turn, leads to various useful properties that make this model more tractable. The study of the randomized version began in the 1980s (see~\cite{SteelMST, SteeleSubadditive, lueker1981optimization, Avram92} among others) with particular emphasis on the Euclidean setting. In 1985, Frieze~\cite{FriezeMST} showed that on $K_n$ with nonnegative i.i.d.~weights on the edges, the expectation of the total weight of the MST is equal to $\zeta(3)/D$ where $D$ is the value of the derivative of the common distribution function of the weights at $0$. The first distributional convergence result for this model was shown a decade later by Janson~\cite{JansonMSTCLT}, who showed that on $K_n$ with Uniform$(0,1)$ or $\Exp(1)$ distributed weights the total weight of the MST obeys the following CLT 
$$
\sqrt{n}(W_n-\zeta(3))\Rightarrow \N(0,6\zeta(4)-4\zeta(3)).
$$
For discussion on the constant that appears in the variance see~\cite{JansonMSTCLTadd,wastlund2005evaluation}. To our knowledge, the rate of convergence in this CLT remains open.

In non-i.i.d.~cases CLTs for MST on the complete graph with weights given by the distances of $n$ uniform points in $[0,1]^d$, for $d\ge 2$, and on Poisson point process on $[0,n^{1/d}]^d$ was shown by Kesten and Lee~\cite{KestenLee96} via a martingale argument. Around the same time, Alexander~\cite{AlexanderMST} proved CLT for the Poissonized version of the problem when $d=2$ via the percolation theoretic approach suggested earlier in~\cite{ramey1982non}. Later, various related results were shown, including CLT for the number of vertices in the MST of a certain degree and for the independence number~\cite{LeeMST1,LeeMST2,LeeIndep}. In 2015, Chatterjee and Sen~\cite{ChatterjeeSenMST, SouravSurvey} used Stein's method to derive the rate of convergence in CLT for the weight of the MST on Poisson points and subsets $d$-dimensional lattice.
For applications of our result to this model, see Section~\ref{sec: MST}.

\subsubsection{Minimal $k$-factor problem}
Minimal $k$-factor problem is a natural generalization of MMP and TSP. Given a graph $G=(V, E)$ with an even number of vertices, and a sequence of edge-weights  $\{\go_e\}_{e\in E}$ the $k$-factor problem asks to find a $k$-regular subgraph of $G$ such that the sum of weights on the edges of the selected subgraph is minimized. Notice that we do not require connectivity of the subgraph, and thus, when $k=1$, this problem reduces to MMP. One could impose the connectivity condition in which case $k=2$ would have corresponded to TSP. To make computations easier, we decided to focus on the former case, although we believe that a similar analysis holds for connected $k$-factor. Understanding the limiting distribution of the minimal $k$-factor is important for hypothesis testing in the planted $k$-factor problem. We refer the interested reader to~\cite{Sicuro_kfact} for more details on the planted $k$-factor problem.
For applications of our result to this model see Section~\ref{sec:kfact}.

\subsubsection{Two main properties}

Despite the broad applicability of our setup to numerous well-known optimization problems (\eg\ those presented above), certain classes of problems lie beyond the scope of this work. We mainly require a problem to have two traits in order to fall into our setup, namely that \textit{every graph in the family over which one optimizes has the same number of edges} and that \textit{any partial matching can be extended to such a graph}. For instance, every spanning tree on $G=(V,E)$ contains $\abs{V}-1$ edges, and every partial matching can be extended to such a tree, yielding that MSTP satisfies these properties. On the other hand, many models do not possess at least one of these properties. For instance, the first passage percolation aims to find a path between two given vertices with the smallest combined weight. Clearly, such paths may contain different number of edges. Similarly, the maximum cut problem does not specify the sizes of allowed cuts over which one optimizes, thus failing the first property. The minimal star problem, which asks to find a spanning star subgraph of minimal weight fails the second required property because no star subgraph can contain two vertex-disjoint edges. We discuss the necessity and approximate versions of these properties in Section~\ref{sec future}.

\subsection{Highlight of the main results in application to MMP on $K_{n,n}$}\label{intro highlight}
Before proceeding to general results, we first present their consequences for the minimal matching problem on the complete bipartite graph $K_{n,n}$. We decided to focus on this problem in this highlight because it is a better-studied model and simpler to analyze at finite temperatures. 

Let $E$ be the set of edges of $K_{n,n}$ and $\{\go_e\}_{e\in E}$ be i.i.d.~random variables from some \textit{general distribution} such that for all $\gb \in [0,\infty)$. In particular, $\go_e$ may take negative values, a case rarely covered by the results present in the literature.
\begin{align*}
     \psi(\gb):=\log\E e^{-\gb\go_e}<\infty.
\end{align*}
We define 
\begin{equation}\label{eq:v}
    v_\gb^2:=\E e^{-2\gb\go_e-2\psi(\gb)}-1=e^{\psi(2\gb)-2\psi(\gb)}-1.
\end{equation}

Notice that under the imposed restrictions on the distribution of the weights $\omega$, we allow them to be infinite with positive probability. This will be used to derive the annealed version of our main results on random graphs; see Section~\ref{sec:random graphs}.

In the spirit of the statistical physics approach, we consider the Gibbs measure on the space of all perfect matchings with the energy given by the total weight of the matching. Let $\cM$ be the set of all perfect matchings in $K_{n,n}$. For any $\pi\in \cM$ define its total weight as 
\[
W(\pi):=\sum_{e\in\pi} \go_e = \sum_{e\in E}\go_e\1_{e\in\pi}.
\]
Then given the weights on edges $\mvgo=(\go_e)_{e\in E}$ and $\gb \in [0,\infty)$
\[
\pr_{\gb,\mvgo}(\pi)=\frac{1}{Z(\gb,\mvgo)}\cdot\frac{1}{n!}\exp\left(-\gb W(\pi)\right),
\]
where $Z(\gb,\mvgo)$ is the normalized constant known as the partition function. We first establish probabilistic limit theorems for the log-partition function, commonly known as free energy.
\begin{thm}\label{thm intro MMP free energy} In the setup as above for any fixed $\gb\in [0,\infty)$ the following limits hold as $n\to \infty$,
    \begin{align*}
        \log Z(\gb,\mvgo) - n\psi(\gb) &\Rightarrow \N\left(-v_\gb^2/2,v_\gb^2\right).
    \end{align*}
\end{thm}

Next, we analyze the structure of the intersections of the independent samples from the Gibbs measure.
\begin{thm}\label{thm intro MMP overlap}
    In the setup as above for any fixed $\gb\in [0,\infty)$ let $\pi_1$ and $\pi_2$ be independent samples from $\pr_{\gb,\mvgo}$, then the number of common edges between $\pi_1$ and $\pi_2$, denoted by $\abs{\pi_1\cap\pi_2}$, converges to $\poi\left(e^{\psi(2\gb)-2\psi(\gb)}\right)$ in probability. 
\end{thm}

In particular, Theorem~\ref{thm intro MMP overlap}  implies replica symmetry, see Remark~\ref{rem:repsym}.

\begin{rem}[$\abs{\pi_1\cap\pi_2}$ at $\gb=0$]\label{indepintersections}
    Theorem~\ref{thm intro MMP overlap} states that the common number of edges between two matchings sampled independently from $\pr_{\gb,\mvgo}$ converges to $\poi(e^{\psi(2\gb)-2\psi(\gb)})$ in probability. For $\gb=0$ the limiting distribution becomes $\poi\left(1\right)$. On the other hand, when $\gb=0$ the statement reduces to the number of common edges between two perfect matchings on $K_{n,n}$ sampled independently and uniformly at random. This in turn can be converted to the question of the number of fixed points in a uniform permutation of $[n]$. Classically, it is known to converge in distribution to $\poi\left(1\right)$, with various proofs available in the literature including the one via the size-bias approach of the Stein--Chen method~\cite[Example 4.21]{RossSurvey}.
    The same trick of rewriting the problem in terms of the number of fixed points of a uniform permutation yields that the size of the intersection of two Hamiltonian cycles on $K_{n}$ is asymptotically $\poi\left(2\right)$. It is not obvious how to apply the same idea in the case of MST. However, one can still derive the analogous result for MST on $K_{n}$ via the Stein--Chen method using a different approach as shown in Theorem~\ref{thm:MSTuniform}.
\end{rem}

We further establish probabilistic limit theorems for the weight of a typical matching under the Gibbs measure and the Gibbs average, \ie~the expectation with respect to the Gibbs measure given the random weights on the edges.

\begin{thm}\label{thm intro MMP typical}
    In the setup as above for any fixed $\gb\in [0,\infty)$ and $\pi$ sampled from $\pr_{\gb,\mvgo}$ the following limits hold as $n\to \infty$
    \begin{align*}
         \frac{1}{\sqrt{n}}\left(W(\pi)+n\psi'(\gb)\right)&\Rightarrow \N\left(0,\psi''(\gb)\right)\text{ and }
         \gibbs{W(\pi)}_\gb+n\psi'(\gb)\Rightarrow \N(\mu,\gs^2),
    \end{align*}
    where
        $$\mu=\left(\psi'(\gb)-\psi'(2\gb)\right)e^{\psi(2\gb)-2\psi(\gb)}$$
        
        and 
        
        $$\gs^2=\left((\psi'(\gb)-\psi'(2\gb))^2+\psi''(2\gb)\right)e^{\psi(2\gb)-2\psi(\gb)}.$$
\end{thm}     

\begin{rem}
    Since $\frac{\partial}{\partial \gb} \log Z(\gb,\mvgo)=\E_{\gb,\mvgo}-W(\pi)=:\gibbs{-W(\pi)}_\gb$ one might expect that Theorem~\ref{thm intro MMP typical} would follow from Theorem~\ref{thm intro MMP free energy} by a simple differentiation argument. Although the means do match in such a fashion, fluctuations are more subtle. In fact, one of the smaller order error terms that would have been missed in the differentiation heuristic will end up contributing a term involving $\psi''(2\gb)$ inside of the variance. Hence, our proof relies on a careful analysis and cluster expansions for $Z(\gb)$.
\end{rem}

\subsection{Setup and Assumptions} \label{intro setup}

This section introduces the Gibbs measure and normalized partition function of an abstract optimization problem.
Let $G_n=(V_n, E_n)$ be a sequence of graphs and $\cS$ be the family of subgraphs of $G$ over which one wants to optimize such that the following assumption holds.
\begin{ass}\label{as:Graph}
 Assume that $G_n=(V_n,E_n)$ is a sequence of edge-transitive graphs, $\abs{V_n}=c\cdot n$ for some $c\in \bN$, and $\abs{E_n}=\Theta(\abs{V_n}^2).$
\end{ass}

To lighten the notation, we often will omit the subscript $n$ if it is clear from the context and the explicit dependence on $n$ is not essential to highlight.
\begin{ass}\label{as:weights}
    We assume that $\{\go_e\}_{e\in E}$ be i.i.d.~random variables from some \textit{general distribution} such that for all $\gb \in [0,\infty)$ we have
\begin{align}\label{eq:psi}
     \psi(\gb):=\log\E e^{-\gb\go}<\infty.
\end{align}
\end{ass}

Given the set $\cS:=\cS_n$ of subgraphs of $G_n$ over which one optimizes. Let $\wt{\cS}:=\wt{\cS}(n)$ denote a collection of all possible subgraphs of elements in $\cS$. For example, for the MMP, $\cS$ is the set of all perfect matchings, while $\wt\cS$ is the set of all partial matchings. We often think of graphs $\gC\in \wt\cS$ (which we call \textit{clusters}) as subgraphs of $G$ that could be extended to a valid configuration (\ie\ an element of $\cS$) for a given optimization problem.

\begin{ass}\label{as:optimization}  
We assume that
\begin{enumerate}[label=\textnormal{III.\arabic*},ref=III.\arabic*]
    \item\label{as:kappa} The edge sets of all elements of $\cS$ are of the same size $m$ and
        
        \begin{equation}\label{eq:gc}
        \lim_{n\to\infty}\frac{m^2}{2\abs{E}}= \gamma\in (0,\infty).
        \end{equation}
        
    \item\label{as:matching} For any vertex-disjoint set of edges (\ie~a partial matching) $\gC\subseteq E_n$ there is $\pi\in \cS$ such that $\gC\subseteq \pi$. 
    \item\label{as:prod of edges} If $\pi$ is chosen uniformly at random from $\cS$ then there is a constant $\rho\in(0,\infty)$ such that for any $k\le \bN$ and any partial matching $\gC$ with $k$ edges we have that
    \begin{equation*}
    \abs{\frac{\pr_0(\gC\subseteq \pi)}{\prod_{e\in\gC}\pr_0(e\in \pi)}-1}\le \frac{\rho^k}{\sqrt{m}}.
    \end{equation*}
    \item\label{as:non matching}
    If $\pi$ is chosen uniformly at random from $\cS$ then there is a constant $\rho\in(0,\infty)$ such that for any $k\in \bN$ and any $\gC\in \wt\cS$ with $k$ edges that is not a partial matching we have that
    $$
    \pr_0(\gC\subseteq \pi)\le \rho^k\prod_{e\in\gC}\pr_0(e\in \pi).
    $$
\end{enumerate}
\end{ass}

Notice that due to the edge transitivity of the graph $G$ for any edge $e\in E$ the probability that it is contained in a uniformly chosen configuration $\pi\in\cS$ is the same and equal to
\begin{equation}\label{eq:p}
    p:=p_n(e):=\pr_0(e\in \pi)=\frac{m}{\abs{E}}.
\end{equation}
This quantity will play an essential role in all of the clusters used in this paper.

Assumption~\ref{as:matching} states that any partial matching can be extended to a graph that belongs to a set $\cS$. Intuitively, it means that $\cS$ is a family of subgraphs that are not necessarily concentrated on one particular part of the graph, as in the latter case, one would expect a wildly different behavior (see Section~\ref{sec:star}).

\begin{rem}
    In general, the assumption of edge-transitivity on the graph (Assumption~\ref{as:Graph}) is not necessary, and we expect that one could carry out the analysis with individual probabilities $p_n(e)=\pr_0 (e\in \pi)$ under certain uniformity assumption. However, for simplicity we decided to focus on the symmetric case.
\end{rem}

\subsection{Notations}\label{intro notations}
Throughout this paper, we adopt the following notations.
\begin{itemize}
    \item $G_n=(V_n,E_n)$ -- a sequence of edge-transitive graphs, $\abs{V}=c\cdot n$ for some $c\in \bN$.
    \item $\cS:=\cS(n)$ -- a collection of spanning subgraphs of $G$ over which we optimize, all of the same size $m$.
    \item $m$ -- the size of the graphs in $\cS$.
    \item $\gamma:=\lim_{n\to\infty} m^2/(2\abs{E})$.
    \item $p:=p_n(e):=\pr_0(e\in \pi)=m/\abs{E}$ -- the probability that an edge $e\in E$ is present in a uniformly chosen configuration $\pi$ from $\cS$.
    \item $\wt{\cS}:=\wt{\cS}(n)$ -- a collection of all possible subgraphs of graphs in $\cS$.
    \item $\cM:=\cM(n)$ -- the collection of perfect matchings of $G$.
    \item $\wt{\cM}:=\wt{\cM}(n)$ -- the collection of partial matchings of $G$.
    \item $\{\go_e\}_{e\in E_n}$ -- the collection of i.i.d.~random variables on edges.
    \item $\xi_e:=\xi_e(\gb):=e^{-\gb\go_e-\psi(\gb)}-1$.
    \item $ v_\gb^2:=\E\xi_e^2=\E e^{-2\gb\go_e-2\psi(\gb)}-1=e^{\psi(2\gb)-2\psi(\gb)}-1$.
    \item $\pr_0(\cdot)$ and $\E_0(\cdot)$ -- the probability and the expectation with respect to a uniformly chosen configuration $\pi$ from $\cS$.
    \item  $Z(\gb):=Z(\gb,\mvgo):=\abs{\cS}^{-1}\cdot\sum_{\pi\in\cS}\exp\left(-\gb W(\pi)\right)$ -- the partition function.
    \item $\wh{Z}(\gb):=\wh{Z}(\gb,\mvgo)=Z(\gb)/\E Z(\gb)$ -- the normalized partition function.
    \item $\gibbs{\cdot}_\gb$ -- the Gibbs average, \ie\ the expectation with respect to Gibbs measure $\pr_{\gb,\mvgo}$ conditioned on the underlying disorder of the weights $\mvgo$.
    \item For a subset of edges $\gC\subseteq E$, we define  $\xi_\gC:=\xi_\gC(\gb):=\prod_{f\in \gC}\xi_f(\gb).$
    \item We will use the big-$O$ and small-$o$ notation as in the usual definition. The symbol $\lesssim$ will also be used instead of the big-$O$ notation.
    \item The notation $X_n = O_{\mathrm{p}}(a_n)$ means that the sequence $X_n/a_n$ is stochastically bounded uniformly in $n$; that is, for any $\eps> 0$, there exists $M_\eps > 0$ such that $\pr(\abs{X_n}/{a_n} > M_\eps) < \eps$ for all sufficiently large $n$. In contrast, $X_n = \op(a_n)$ means that $X_n/a_n$ converges to zero in probability.
    \item By $X_n\Rightarrow Y$ we mean that the random variables $X_n$ converge weakly to $Y$.
\end{itemize}
To lighten the notation, we often omit the subscripts $n$ and $0$ when the context makes it clear, and the explicit dependence on $n$ is not essential to highlight. We will also abuse the notation and write $E$ (or $V$) for both the edge set (or vertex set) as well as the size, \eg\ ``$e\in E$" and ``summing over all clusters of the size $k\le E$."

\subsection{Heuristic}\label{intro heuristic}
We consider the Gibbs measure on the set $\cS$ with the inverse temperature parameter $\gb\ge 0$ and the weight of a configuration $\pi\in \cS$ 
\[
W(\pi):=\sum_{e\in\pi} \go_e = \sum_{e\in E}\go_e\1_{e\in\pi},
\]
defined by
\[
\pr_{\gb,\mvgo}(\pi)=\frac{1}{ Z(\gb,\mvgo)}\cdot \frac{1}{\abs{\cS}}\exp\left(-\gb W(\pi)\right),\quad \pi\in\cS.
\]
The  partition function $Z(\gb)$ is defined as
\begin{equation}\label{eq:Z}
    Z(\gb):=Z(\gb,\mvgo):=\frac{1}{\abs{\cS}}\sum_{\pi\in\cS}\exp\left(-\gb W(\pi)\right)
\end{equation}
and normalized partition function $\wh{Z}(\gb):=Z(\gb)/\E Z(\gb)$. Note that 
$
\E Z(\gb)=e^{m\psi(\gb)}
$
and recall that $\pr_0$ and $\E_0$ denote the probability and the expectation with respect to a uniformly chosen configuration $\pi$ from $\cS$. 

Now one can see that 
\begin{align}\label{eq:whZ}
    \wh{Z}(\gb)&=\E_0 \prod_{e\in E}e^{(-\gb\go_e-\psi(\gb))\1_{e\in\pi}}=\E_0 \prod_{e\in E}\left(1+\left(e^{-\gb\go_e-\psi(\gb)}-1\right)\1_{e\in\pi}\right).
\end{align}

Define the mean zero random variables for each $e\in E$,
\begin{equation}\label{eq:xi}
\xi_e=\xi_e(\gb):=e^{-\gb\go_e-\psi(\gb)}-1,
\end{equation}
then~\eqref{eq:whZ} can be rewritten as
\begin{equation}\label{eq:whZ2}
    \wh{Z}(\gb)=\E_0 \prod_{e\in E} (1+\xi_e\1_{e\in\pi}).
\end{equation}
We can formally expand the product as sums over all subgraphs $\gC\subseteq E$ and get
\begin{align}\label{eq:pse}
    \wh{Z}(\gb)= \sum_{\gC\subseteq E} \pr_0(\gC\subseteq \pi)\cdot \prod_{e\in \gC}\xi_e.
\end{align}
Moreover, the form of $\wh{Z}(\gb)$ in~\eqref{eq:whZ2} allows us to conveniently rewrite its second moment in terms of the size of expected intersection of two independent samples $\pi,\pi'$ from $\pr_0$, below we use slight abuse of notation to denote the expectation with respect to both $\pi$ and $\pi'$ by $\E_0$.
\begin{align}\label{eq:whZ^2}
    \E\wh{Z}^2(\gb)&=\E\E_0 \prod_{e\in E} (1+\xi_e\1_{e\in\pi}+\xi_e\1_{e\in\pi'}+\xi_e^2\1_{e\in\pi\cap\pi'})\notag\\
    &=\E_0 \prod_{e\in E}(1+\E\xi_e^2\cdot \1_{e\in\pi\cap\pi'})
    =\E_0 \prod_{e\in E}(1+v_\gb^2\cdot \1_{e\in\pi\cap\pi'}),
\end{align}
where we use that $\{\xi_e\}_{e\in E}$ are mean zero random variables with variance $v_\gb^2$. We can recombine terms to get
\[
\E\wh{Z}^2(\gb)=\E_{0} e^{\log(1+v_\gb^2)\ \abs{\pi\cap\pi'}}.
\]
Under our assumptions, one can show that $\abs{\pi\cap\pi'}$ is asymptotically a Poisson distributed random variable with mean $|E|\cdot p^2=m^2/|E|\approx 2\gc$, see Remark~\ref{indepintersections}. Hence, we expect 
$
\E\wh{Z}^2(\gb)\approx e^{2\gc v_\gb^2}<\infty.
$
This allows us to use the expansion~\eqref{eq:pse} of $\wh{Z}$ as an $L^2$-convergent power series in terms of the size of $\gC$ and enable the cluster expansion technique. 

\subsection{Cluster expansion}\label{intro cluster}
Cluster expansion originated from the work of Mayer and Montroll in 1941, which related to expansions for a dilute gas, and played an important role in the mathematical work of statistical mechanics, condensed matter physics, and quantum field theory. The methods have been successfully used in lattice models, including the Ising model~\cite{ScolaClust}, random cluster model \cite{HelmuthClust}, and self-avoiding random walk \cite{KhaninClust}, among others. For disordered models, cluster expansion has been employed in random polymer and spin-glass models to understand the limiting behavior of the log-partition function (see, e.g., \cite{AizenmanClust,BanerjeeClust,ParthaQiang,AlbertsKhaninQuastel} among others). More recently, this technique was applied in combinatorics to calculate the number of certain objects and to gain insight into their structural properties (\cite{JenssenClust,JenssenClustTorus} among others). We refer the interested reader to~\cite{fvbook, FarisCluster} for more details about cluster expansion techniques in statistical mechanics and connections to enumerative combinatorics. 

Cluster expansion is a formal power series expansion of the partition function in terms of a sum over finitely many interactions. We present two classical versions of such expansion. 
For many statistical physics models, including the models considered in this paper, the energy corresponding to a configuration $\pi\in \cS$ can be written as a linear combination of simpler $\{0,1\}$--valued functions of $\pi$, \ie\
\[
W(\pi)=\sum_{e\in E} a_{e} f_{e}(\pi).
\]
Thus, we get
\begin{align*}
\exp(-\gb W(\pi)) =\prod_{e\in E} \exp( -\gb a_{e} f_{e}(\pi)) 
&= \prod_{e\in E} \left(1+(e^{-\gb a_e}-1)\cdot f_{e}(\pi)\right) \\
&=\sum_{\gC \subseteq E} \prod_{e\in \gC} (e^{-\gb a_e}-1)\cdot \prod_{e\in \gC}  f_{e}(\pi) 
\end{align*}
and for the partition function, we get 
\begin{align}\label{eq:01exp}
    Z(\gb)=\sum_{\gC \subseteq E} \prod_{e\in \gC} (e^{-\gb a_e}-1)\cdot  \frac1{\abs{\cS}}\sum_{\pi\in\cS}\prod_{e\in \gC}  f_{e}(\pi).
\end{align}
These types of decompositions are called \textit{cluster expansions, where subgraphs $\Gamma$ are referred to as \textit{clusters}.}
When $a_e\le 0$ for all $e\in E$, the terms in~\eqref{eq:01exp} are all non-negative, and the high temperature regime corresponds to when the sum converges absolutely. However, for the disordered model, $a_e$'s can take both positive and negative values, and absolute convergence is, most of the time, not the correct way to control the growth of the partition function. 

On the other hand, if in a disordered model $f_{e}(\cdot)\in \{\pm1\}$ and $a_{e}$'s are  independent symmetric random variables, such as spin glass models, one can write
\begin{align*}
\exp(-\gb W(\pi))&=\prod_{e\in E} \exp( -\gb a_{e} f_{e}(\pi)) = \prod_{e\in E} \cosh(\gb a_{e})(1+\tanh(- \gb a_{e})\cdot f_{e}(\pi)) \\
&=\prod_{e\in E} \cosh(\gb a_{e}) \cdot \sum_{\gC\subseteq \E} \prod_{e\in \gC} \tanh(-\gb a_{e})\cdot \prod_{e\in \gC}  f_{e}(\pi).
\end{align*}
Thus, we get the following chaos decomposition in terms of sums of finite degree polynomials of the independent mean zero random variables $(\tanh( \gb a_e))_{e\in E}$,
\begin{align}\label{eq:chaos1}
    \frac{1}{\prod_{e\in E} \cosh(\gb a_{e})}\cdot Z(\gb)=  \sum_{\gC\subseteq E} \prod_{e\in \gC} \tanh(-\gb a_{e}) \cdot  \sum_{\pi\in \cS}\prod_{e\in \gC}  f_{e}(\pi).
\end{align}

For random optimization problems, we have $f_{e}(\pi)=\ind(e\in\pi)\in \{0,1\}$ and $a_e$'s are independent random variables that follow a general distribution. Thus, neither of the expansions outlined above is applicable to study the asymptotic behavior of the partition function. There are a few examples in the literature on cluster expansion that truly lie outside of the scope of these two approaches. For example, the directed polymer in dimensions $(1+1)$ and $(2+1)$ in the intermediate disorder regime~\cite{AlbertsKhaninQuastel,CaravennaSunZygouras,DZ16}. However, in such cases, the authors considered the case where $\beta_n\to 0$ sufficiently fast, which again does not apply to our setting. 

A technical novelty of the present paper is the extension of both the approaches from above, which enables us to derive the expansion~\eqref{eq:whZ}. To our knowledge, our expansion is the first chaos--type expansion for disordered combinatorial problems under the general setup. We would like to highlight that the key assumption that enables our analysis is that $\abs{\pi}=m$ for all $\pi\in\cS$ (see Assumption~\ref{as:kappa}).


\section{Main results}\label{sec mainres}
In this section we state main results in the general setting introduced in Section~\ref{intro setup}.
\begin{lem}\label{lem:Zhat}
Under Assumptions~\ref{as:Graph}--\ref{as:optimization} for any $\gb\in [0,\infty)$
\[
\E \abs{\prod_{e\in E}\left(1+p\xi_e(\gb)\right)-\wh{Z}(\gb)}^2\lesssim \frac{1}{m}.
\]
\end{lem}
This approximation yields a CLT for the log-partition function.
\begin{thm}[CLT for the log-partition function]\label{thm:log part clt}
Under Assumptions~\ref{as:Graph}--\ref{as:optimization}, for any $\gb\in [0,\infty)$, we have
\begin{equation}\label{eq:CLT for wh{Z}}
    \log Z(\gb)-m\psi(\gb)\Rightarrow \N\left(-\gc v_\gb^2,2\gc  v_\gb^2\right) \text{ as } n\to\infty,
\end{equation}
where $\gc \in(0,\infty)$ is given by~\eqref{eq:gc}.
\end{thm}
An immediate corollary of Theorem~\ref{thm:log part clt} is the law of large numbers for the log-partition function. 
\begin{cor}\label{thm:Limiting free energy}
    Under Assumptions~\ref{as:Graph}--\ref{as:optimization}, for any $\gb\in [0,\infty)$ we have
    \[
    \frac{1}{m}\log Z(\gb,\mvgo)\to\psi(\gb) \quad\text{ and }\quad  \frac{1}{m}\gibbs{W(\pi)}_\gb\to -\psi'(\gb)
    \]
    in probability.
\end{cor}
\begin{proof} By definition 
    $$
    \frac{1}{m}\log Z(\gb,\mvgo)=\psi(\gb)+\frac{1}{m}\log\wh{Z}(\gb,\go).
    $$
    By Theorem~\ref{thm:log part clt}, CLT holds for $\log\wh{Z}(\gb,\go)$. Thus, $m^{-1}\log\wh{Z}(\gb,\go)$ is negligible in probability.
    Note that 
    \begin{align*}
        \frac{\partial}{\partial \gb} \log Z_n(\gb,\mvgo)&=\gibbs{-W(\pi)}_\gb\qquad \textnormal{and}\qquad
        \frac{\partial^2}{\partial \gb^2} \log Z_n(\gb,\mvgo)=\left(\gibbs{W(\pi)}^2_\gb-\gibbs{W(\pi)}_\gb\right)^2.
    \end{align*}
    In particular, the log-partition function is convex. Thus, the convergence of $\frac{1}{m}\log Z(\gb,\mvgo)$ to $\psi(\gb)$ implies the convergence of the first derivative to that of the limit in $\beta$. 
\end{proof}

\begin{thm}[Size of the intersection of two Gibbs samples]\label{thm:poi limit}
    Under Assumptions~\ref{as:Graph}--\ref{as:optimization}, for any $\gb\in [0,\infty)$, let $\gc$ be defined as in~\eqref{eq:gc}, and $\pi$ and $\pi'$ be two independent samples from the Gibbs measure $\pr_{\gb,\mvgo}$ then the number of common edges between $\pi$ and $\pi'$, denoted by $\abs{\pi\cap\pi'}$, converges to $\poi\left(2\gc\cdot e^{\psi(2\gb)-2\psi(\gb)}\right)$ in probability.
\end{thm}

\begin{rem}\label{rem:repsym}
In particular, Theorem~\ref{thm:poi limit} implies that this model is replica symmetric in a strong sense, as the overlap, defined as $\frac{1}{n}\abs{\pi_1\cap\pi_2}$, is $O_{\mathrm{p}}(1/n)$. 
\end{rem}

\begin{rem}[Size of the intersection of three Gibbs samples]\label{rem: 3replica}
     In fact, one can show that under Assumptions~\ref{as:Graph}--\ref{as:optimization} for any $\gb\in [0,\infty)$ if $\pi_1$, $\pi_2$ and $\pi_3$ are three independent samples from the Gibbs measure $\pr_{\gb,\mvgo}$ then the number of common edges between them is zero with high probability. One way of proving this is to argue by the first-moment method. Indeed, using a similar, but more involved, cluster expansion to the one presented in the proof of Theorem~\ref{thm:poi limit} one can show that for any edge $e\in E$
     \[
     \pr_{\gb,\mvgo}(e\in \pi_1)=p\cdot \frac{1+\xi_e(\gb)}{1+p\cdot \xi_e(\gb)}(1+\op(1))\approx \frac{1}{n}.
     \]
\end{rem}

\begin{thm}[Quenched CLT for a typical configuration]\label{thm:typical CLT}
    Under Assumptions~\ref{as:Graph}--\ref{as:optimization} for $\pi\sim\pr_{\gb,\go}$ and  any $\gb\in [0,\infty)$  we have that
    \[
    \frac{1}{\sqrt{m}}\left(W(\pi)+m\psi'(\gb)\right)\Rightarrow \N\left(0,\psi''(\gb)\right)
    \]
    with high probability in $\mvgo$.
\end{thm}
\begin{cor}[Quenched LLN for a typical configuration]\label{thm:typical LLN}
    Under Assumptions~\ref{as:Graph}--\ref{as:optimization} for $\pi\sim\pr_{\gb,\go}$ and  any $\gb\in [0,\infty)$  we have that
    \[
    \frac{1}{m}W(\pi)\to -\psi'(\gb),
    \]
    in probability in $\pr_{\gb,\go}$ with high probability in $\mvgo$. 
\end{cor}

Notice that the limits in Theorem~\ref{thm:typical CLT} and Corollary~\ref{thm:typical LLN} depend only on the weights $\mvgo$ and not on the underlying graph as long as Assumption~\ref{as:Graph} is satisfied. Thus the annealed versions of these results follow immediately. Finally, we derive the central limit theorem for the Gibbs average.

\begin{thm}[CLT for the Gibbs average]\label{thm:gibbs CLT}
    Under Assumptions~\ref{as:Graph}--\ref{as:optimization} for $\pi\sim\pr_{\gb,\go}$ and  any $\gb\in [0,\infty)$ we have that $\gibbs{W(\pi)}_\gb+m\psi'(\gb)$ weakly converges to normal distribution with mean $$2\gc\left(\psi'(\gb)-\psi'(2\gb)\right)e^{\psi(2\gb)-2\psi(\gb)}$$ and variance $$2\gc \left((\psi'(\gb)-\psi'(2\gb))^2+\psi''(2\gb)\right)e^{\psi(2\gb)-2\psi(\gb)}.$$
\end{thm}

\section{Extensions of main results}\label{sec:extensions of main res}
\subsection{Application to multipartite graphs}\label{sec:block models}
Although it is plausible that random optimization problems on multipartite graphs satisfy our Assumptions~\ref{as:Graph}--\ref{as:optimization}, verifying them turns out to be involved beyond complete graph and complete bipartite graph. Here we present another way of approaching optimization problems on these graphs that generalizes our method. One of the main assumptions for the mean-field computations, namely Assumption~\ref{as:kappa}~will be different now. For more discussion see Section~\ref{sec:genblock}.

Consider a complete $\ell$--partite graph of size $n$ with $\ell\ge2$ many blocks of sizes $n_1,n_2,\ldots,n_\ell$, respectively, where $n_i/n \to \gl_i>0, i=1,2,\ldots,\ell$ and $\gl_1+\gl_2+\cdots+\gl_\ell=1$. Let $E_{s,t}$ be the set of edges between block $s$ and $t$, for $1\le s\le t\le \ell$. Clearly, 
\[
|E_{s,t}|=\begin{cases}
n_{s}n_{t} & \text{ for } s<t,\\
n_{s}(n_{s}-1)/2 & \text{ for } s=t.
\end{cases}\approx 
\begin{cases}
\gl_{s}\gl_{t}\cdot n^{2} & \text{ for } s<t,\\
\frac12\gl_{s}^{2}\cdot n^{2} & \text{ for } s=t.
\end{cases}
\]
We assume the following generalization of Assumption~\ref{as:kappa}.
    
\begin{customthm}{\textnormal{{III.}}$1'$}\label{as:optimizationp}
 For all elements $\pi\in\cS$, we have $\abs{\pi\cap E_{s,t}}=m_{s,t}$ for $1\le s\le t\le \ell$ and
        
        \begin{equation}\label{eq:gcp}
        \lim_{n\to\infty}\frac{m_{s,t}^2}{2\abs{E_{s,t}}}= \gamma_{s,t}\in (0,\infty).
        \end{equation}
\end{customthm}

Denoting the vector $\vm=(m_{st})_{s\le t}$, sometimes we will write $\cS_{\vm}$ instead of $\cS$ to emphasize the dependence on $\vm$. Note that, for a uniformly at random chosen configuration $\pi$ from $\cS_{\vm}$ we have 
\[
p_{s,t}:=\pr_0(e\in \pi) = \frac{m_{s,t}}{|E_{s,t}|}\approx \sqrt{\frac{2\gc_{st}}{|E_{st}|}}\text{ for an edge } e\in E_{s,t}, s\le t.
\]
Moreover, for two i.i.d.~uniformly at random $\pi,\pi'$ from $\cS_{\vm}$, we expect $$\abs{\pi\cap\pi'\cap E_{s,t}} = \sum_{e\in E_{s,t}} \1_{e\in \pi, e\in \pi'}$$
 to have asymptotically independent Poisson distribution with mean $|E_{s,t}|\cdot p_{s,t}^{2}\approx 2\gc_{s,t}$ for $s\le t$, under Assumptions~\ref{as:matching},~\ref{as:prod of edges}, and~\ref{as:non matching}.

For $s\le t$, we assume that the distribution of the random edge-weight $\go_{e}$ for an edge $e\in E_{s,t}$ depends only on $(s,t)$ and we denote a typical random weight by $\go_{st}$. Assume that $\go_{st}$ satisfies 
\[
\psi_{s,t}(\gb):=\log \E e^{-\gb \go_{st}}<\infty \text{ for } \gb>0.
\]
Define, 
\[
v_{\gb}^{2}(s,t):=\exp(\psi_{s,t}(2\gb)-2\psi_{s,t}(\gb)) -1, s\le t.
\]
A similar computation for 
\[
Z_{n}(\gb):=\frac{1}{\abs{\cS_{\vm}}}\sum_{\pi\in\cS_{\vm}} \exp\left(-\gb \sum_{e\in\pi} \go_{e}\right),
\]
shows that $\E Z_{n}(\gb) = \exp(\sum_{s\le t} m_{s,t} \psi_{s,t}(\gb))$ and for $\wh{Z}_{n}(\gb)=Z_{n}(\gb)/\E Z_{n}(\gb)$, we have
\[
\E \wh{Z}_n(\gb)^2 = \E_0 \prod_{s\le t} \exp(\ln(1+v_{\gb}^{2}(s,t))\cdot \abs{\pi\cap\pi'\cap E_{s,t}}) \approx e^{2\sum_{s\le t}\gc_{st}v_{\gb}^{2}(s,t)}.
\]
In particular, the following generalization of Theorem~\ref{thm:log part clt} holds. The proof is given in Section~\ref{sec proofs}.

\begin{thm}[CLT for the log-partition function]\label{thm:log part cltp}
Under Assumptions~\ref{as:weights}, \ref{as:optimizationp}, and~\textnormal{\ref{as:matching}--~\ref{as:non matching}}, for any $\gb\in [0,\infty)$, we have  as $n\to\infty$
\begin{equation*}\label{eq:CLT for wh{Z}p}
    \log Z(\gb)-\sum_{s\le t} m_{s,t} \psi_{s,t}(\gb)\Rightarrow \N\left(-\sum_{s\le t}\gc_{st}v_{\gb}^{2}(s,t),2\sum_{s\le t}\gc_{st}v_{\gb}^{2}(s,t)\right).
\end{equation*}
\end{thm}

\subsection{Application to random graphs}\label{sec:random graphs}
    So far we considered random optimization problems only on deterministic graphs. However, it is natural to study the same problems on random ones. As above it is plausible that variety of random graphs satisfy our assumptions (at least in some approximate form) with high probability. Moreover, Janson~\cite{JansonCLTnumber} has already derived the orders of the mean and variances of the numbers of perfect matchings, Hamiltonian cycles, and spanning trees in Erd\H{o}s--R\'enyi random graph $G_{n,p}$ and in a uniform random graph with $m$ edges $G(n,m)$. He also derived Gaussian and log-Gaussian limiting theorems for appropriately scaled and centered version of these quantities. However, as a consequence of Janson's results one can see that these counts are not concentrated, and thus verifying our Assumption~\ref{as:prod of edges} requires further careful analysis.
    
    In this section we present a way of deriving annealed results (with respect to randomness of the graph) directly from our main results.
    Note that, the only assumption put on the distribution of weights was that $\log\E e^{-\gb\go}<\infty$. Thus, our main results directly yield the annealed versions when the underlying graph is a homogeneous random graph such as Erd\H{o}s--R\'enyi $G_{n,p}$, or stochastic block model, in general. These models can be constructed by considering $K_n$ or other complete $\ell$-partite graphs, respectively, and deleting every edge independently with probability $1-p$.
    
    Consider an optimization problem on $G_{n,p}$, for some $p\in(0,1)$, and let $\{\go_e\}_{e\in E}$ be the sequence of weights that satisfy Assumption~\ref{as:weights}. One can convert this problem to the same on $K_{n}$ by considering weights $\{\wh{\go}_e\}_{e\in E}$ defined by
    \[
    \wh{\go}_e=
    \begin{cases}
    \go_e   &\text{w.p.}\quad p\\
    \infty   &\text{w.p.} \quad 1-p
    \end{cases}
    \]
    and notice that $\{\wh{\go}_e\}_{e\in E}$ still satisfy Assumption~\ref{as:weights}. This yields
    \[
\psi_p(\gb):=\log(\E e^{-\gb\wh\go}) = \log p + \psi(\gb)<\infty.
    \]
    In such cases, there are three levels of randomness:
    \begin{enumerate}
        \item[a.] randomness of the graph,
        \item[b.] the random edge-weights,
        \item[c.] sampling according to the Gibbs measure. 
    \end{enumerate}
    
    By considering $\wh{\go}_e$ one combines the first two levels, hence the resulting limit theorems are annealed only with respect to the randomness of the underlying graph. Similar approach can also be used for block models.

\begin{obs}\label{obs:random-graphs}
Let $G(n,p)$ denote the Erd\H{o}s--R\'enyi random graph, that is a graph on $n$ vertices where each edge is present with probability $p$ independently of each other. Then all of the results presented in Section~\ref{sec mainres} hold in probability with respect to $G(n,p)$.
\end{obs}
\begin{problem}
    Derive analogous results in quenched form, \ie\ the results that hold for each realization of a random graph with high probability.
\end{problem}

\section{Applications}\label{sec appl}
In this section, we show that our main results apply to all the optimization problems mentioned in Section~\ref{sec: Review}.  

\subsection{Minimal matching problem on $K_{n,n}$, $K_{2n}$, $G_{2n,\wh{p}}$, and $G(2n,\wh{m})$}\label{sec: MMP}
\subsubsection{MMP on $K_{n,n}$}
In terms of the setup from Section~\ref{intro setup} for the MMP on $K_{n,n}$ we have
\begin{itemize}
    \item $G_n=K_{n,n}$,$\abs{V}=2n$, and $\abs{E}=n^2$.
    \item $\cS=\cM$ - the set of all perfect matchings on $G$, 
    \item $m=n$, because each perfect matching contains exactly $\abs{V}/2=n$ edges.
    \item $p=1/n$, and  $\gamma=1/2$.
\end{itemize}
Clearly Assumptions~\ref{as:Graph} and~\ref{as:kappa}~are satisfied. Since any partial matching on $K_{n,n}$ can be extended to a perfect matching Assumption~\ref{as:matching}~is satisfied as well. Assumption~\ref{as:non matching}~trivially holds because in the case of MMP $\wt{\cS}\setminus\wt{\cM}=\emptyset$. Thus, it remains to check  Assumption~\ref{as:prod of edges}, \ie~that for any partial matching $\gC\in\wt{\cM}$ of size $k$ and some $\rho\in(0,\infty)$ we have that

\begin{equation}\label{eq:MMPKnnProd}
\abs{{\pr_0(\gC\subseteq \pi)}/{p^k}-1}\le \frac{\rho^k}{\sqrt{m}}.    
\end{equation}

Denote the falling factorial as $(n)_k:=n!/(n-k)!$. Since $\pr_0(e\in \pi)={1}/{n}$, we can rewrite the left-hand side of~\eqref{eq:MMPKnnProd} as follows
\begin{align*}
    \abs{{\pr_0(\gC\subseteq \pi)}/{p^k}-1}
&=\abs{n^k\pr_0(\gC\subseteq \pi)-1}\\
    &=\abs{\frac{n^k}{(n)_k}-1}
    \le \frac{k(k-1)}{2(n-k+1)}
    = O\left(\frac{\rho^k}{n}\right),   
\end{align*}
for some $\rho\in(0,\infty)$.
 Thus, if $\{\go_e\}_{e\in E}$ are i.i.d.~and are such that \eqref{eq:psi} holds the MMP on $K_{n,n}$ satisfies our assumptions, and hence all of the main results apply.

\subsubsection{MMP on $K_{2n}$}

In terms of the setup from Section~\ref{intro setup} for the minimal matching problem on $K_{2n}$ we have
\begin{itemize}
    \item $G_n=K_{2n}$, $\abs{V}=2n$, and $\abs{E}=\binom{2n}{2}$.
    \item $\cS=\cM$ - the set of all perfect matchings on $G$, 
    \item $m=n$, because each perfect matching contains exactly $\abs{V}/2=n$ edges.
    \item $p=1/(2n-1)$, and $\gamma=1/4$.
\end{itemize}

Again Assumptions~\ref{as:Graph} and~\ref{as:kappa}~are satisfied. Since any partial matching on $K_{2n}$ can be extended to a perfect matching Assumption~\ref{as:matching}~is also satisfied. Again, Assumption~\ref{as:non matching}~trivially holds because $\wt{\cS}\setminus\wt{\cM}=\emptyset$. Thus, it remains to check Assumption~\ref{as:prod of edges}, \ie~that for any partial matching $\gC\subseteq\wt{\cM}$ of size $k$ we have that
\begin{equation}\label{eq:MMPK2nProd}
\abs{{\pr_0(\gC\subseteq \pi)}/{p^k}-1}\le \frac{\rho^k}{\sqrt{m}}.    
\end{equation}
Since 
$
\pr_0(e\in \pi)=1/(2n-1)
$ 
and the total number of perfect matchings on $K_{2n}$ is $\abs{\cM}=(2n)!/2^nn!$, we derive that for any partial matching $\gC\in \wt{\cM}$ of size $k$
\[
\pr_0(\gC\subseteq \pi)=\frac{2^k(2n-2k)!n!}{(2n)!(n-k)!}=\prod_{i=0}^{k-1}\frac{1}{2n-2i-1}.
\]
And thus the left-hand side of~\eqref{eq:MMPKnnProd} can be treated in the following way
\begin{align*}
    \abs{\frac{\pr_0(\gC\subseteq \pi)}{\prod_{e\in\gC}\pr_0(e\in \pi)}-1}&=\abs{\prod_{i=0}^{k-1}\frac{2n-1}{2n-2i-1}-1}\\
    &=\abs{\prod_{i=0}^{k-1}\left(1+\frac{2i}{2n-2i-1}\right)-1}
    \le \frac{k(k-1)}{4(n-k+1/2)}.
\end{align*}
Thus, if $\{\go_e\}_{e\in E}$ are i.i.d.~and are such that \eqref{eq:psi} holds the MMP on $K_{2n}$ satisfies our assumptions and hence all of the main results apply.

\subsection{Traveling salesman problem on $K_n$}\label{sec: TSP}

In terms of the setup from Section~\ref{intro setup} for the traveling salesman problem on $K_n$ we have
\begin{itemize}
    \item $G_n=K_{n}$, $\abs{V}=n$, and $\abs{E}={n\choose 2}$,
    \item $\cS=\cP$ - the set of all Hamiltonian (self-avoiding and spanning) cycles on $G$, 
    \item $m=n$, because each path contains exactly $\abs{V}=n$ edges.
    \item $p=2/(n-1)$, and 
    $\gamma=1$.
\end{itemize}
Hence Assumptions~\ref{as:Graph}~and~\ref{as:kappa}~are satisfied. Since any partial matching on $K_{n}$ can be extended to Hamiltonian path Assumption~\ref{as:matching}~is satisfied as well. It remains to check  Assumptions~\ref{as:prod of edges}~and~\ref{as:non matching},~\ie~if $\wt{\cP}_n$ be the set of all partial Hamiltonian paths of $G$ one has to check that for any partial matching $\gC\in\wt{\cM}\subseteq\wt{\cP}_n$ with $k$ edges we have that
\begin{equation}\label{eq:TSPKnProd}
\abs{{\pr_0(\gC\subseteq \pi)}/{p^k}-1}\le \frac{\rho^k}{\sqrt{m}},
\end{equation}
 for some $\rho\in(0,\infty)$, and for the rest of $\gC\in\wt{\cP}_n\setminus \wt{\cM}$ with $k$ edges we have that
$\pr_0(\gC\subseteq \pi)\le (\rho\cdot p)^{k},$ for some $\rho\in(0,\infty)$.

Notice that the number of Hamiltonian cycles in $K_n$ is $(n-1)!/2.$
This can be seen by encoding cycles as self-avoiding paths of length $n-1$, and then connecting the last vertex to the first, there are $n!$ of those. Since cycles are unoriented and each cycle can be represented as a path started at any of the $n$ vertices, we conclude that each cycle was counted $2n$ times in the $n!$.
Thus
\[
p_e:=\pr_0(e\in\pi)=\frac{2}{n-1}.
\]

Let $\wt{\cP}_n$ be the set of all partial Hamiltonian paths of $G$. 
Let $\gC'\in \wt{\cP}_n$ be a partial path with $k$ edges and $s$ many connected components. If $\pi$ is chosen uniformly at random from $\cP_n$ we get that
\begin{align*}
\pr_0(\gC'\subseteq\pi)&=\frac{2^{s-1} (n-k-1)!}{2^{-1}(n-1)!} 
=\frac{2^{s}(n-k-1 )!}{(n-1)!}\approx \frac{2^{s}}{n^{k}}.
\end{align*}

In particular, since $s\le k$ Assumption~\ref{as:non matching} holds.
To verify Assumptions~\ref{as:prod of edges}, we consider the case when $\gC$ is a partial matching, \ie~$s=k$, and thus the left-hand side of~\eqref{eq:TSPKnProd} can be treated in the following way
\begin{align*}
    \abs{\frac{\pr_0(\gC\subseteq \pi)}{p^k}-1}&=\abs{\frac{2^{k}(n-k-1 )!(n-1)^k}{(n-1)!2^k}-1}\\
    &=\abs{\frac{(n-1)^k}{(n-1)_k}-1}
    \le \frac{k(k-1)}{2(n-k)}
    \le O\left(\frac{\rho^k}{n}\right).    
\end{align*}
Thus, if $\{\go_e\}_{e\in E}$ are i.i.d.~and are such that \eqref{eq:psi} holds the TSP on $K_{n}$ satisfies our assumptions, and hence all of the main results apply.

\subsection{Minimal spanning tree problem on $K_n$}\label{sec: MST}

In terms of the setup from Section~\ref{intro setup} for the Minimal Spanning Tree on $K_{n}$ we have
\begin{itemize}
    \item $G_n=K_{n}$, $\abs{V}=n$, and $\abs{E}={n\choose 2}$,
    \item $\cS=\cT$ - the set of all spanning trees on $G$, 
    \item $\wt{\cS}=\wt{\cT}$ - the set of all spanning forests on $G$, 
    \item $m=n-1$, because each finite tree on $\abs{V}=n$ vertices contains exactly $n-1$ edges.
    \item $p=2/n$, and
     $\gamma=1$.
\end{itemize}

Similarly to the applications above Assumptions~\ref{as:Graph} and~\ref{as:kappa} are clearly satisfied. Since any partial matching on $K_{n}$ can be extended to a spanning tree Assumption~\ref{as:matching} is satisfied as well. Hence it remains to check Assumptions~\ref{as:prod of edges} and~\ref{as:non matching}, \ie~for any partial matching $\gC\in\wt{\cM}\subseteq\wt{\cT}$ with $k$ edges we have that
\begin{equation}\label{eq:MSTPKnProd}
\abs{{\pr_0(\gC\subseteq \pi)}/{p^k}-1}\le \frac{\rho^k}{\sqrt{m}},
\end{equation}
and for the rest of $\gC\in\wt{\cT}_n\setminus \wt{\cM}$ with $k$ edges we have that
\begin{equation}\label{eq:TSPKnProd2}
\pr_0(\gC\subseteq \pi)\le (\rho\cdot p)^{k},
\end{equation}
for some $\rho\in(0,\infty)$.
We first state a generalization of Cayley's formula that counts the number of spanning forests that extend a fixed one.
\begin{lem}\label{th:cayley}
    Let $F$ be a spanning forest on $n$ vertices. Let $t$ be the number of trees in $F$ and $\mvs=(s_1, s_2,\ldots, s_t)$ be the sizes of their vertex sets, respectively. Then the number of spanning trees $N(F)$ that contain $F$ is given by
    \begin{equation}\label{eq:cayley}
    N(F)=N_{\mvs,t}=n^{t-2}\prod_{i=1}^t s_i.
    \end{equation}
\end{lem}
It is easy to see that if $t=n$, then each $s_i=1$, and~\eqref{eq:cayley} recovers classical Cayley's formula that states that there are a total of $n^{n-2}$ spanning trees on $n$ vertices. The proof of this statement is similar to the proof of Cayley's formula via generating functions; we provide it 
here for completeness. 
\begin{proof}
    First, consider each connected component in $F$ as a single vertex and call the set of such vertices $\wh{V}$. Notice that $|\wh{V}|=t$. By Cayley's formula, there are $n^{t-2}$ spanning trees on this set. Given a tree $\wh{T}$ with the vertex set $\wh{V}$ we may retrieve a desired spanning tree on $[n]$ by inserting edges between vertices in the corresponding connecting components, the number of choices of edges that could connect a pair $(i,j)\in E(\wh{T})$ is given by a product of $s_i$ and $s_j$. Hence  
\begin{equation}\label{eq:cayley2}
    N_{\mvs,t}=\sum_{\wh{T}}\prod_{(i,j)\in E(\wh{T})} s_is_j=\sum_{\wh{T}}\prod_{i=1}^t s_i^{\deg_{\wh{T}}(i)},
    \end{equation}
    where $\deg_{\wh{T}}(i)$ denotes the degree of vertex $i$ in $\wh{T}$.
    Next, we prove by induction on the number of trees $t$ that for all $t\in\bN$
    \begin{equation}\label{eq:indforest}
        N_{\mvs,t}=(s_1+s_2+\cdots+s_t)^{t-2}\cdot s_1s_2\cdots s_t.
    \end{equation}

    Base Case: when $t=1$, the forest $F$ is connected hence $s_1=n$ and $N_{\mvs,t}=1$.
    
    Inductive step: assume that the equality~\eqref{eq:indforest} holds for all $\mvs$ at a particular $t\in\bN$. The $(t+1)$-th vertex can be attached at any place in $\wh{T}$. Viewing $N_{\mvs,t+1}$ as a polynomial in the $s_i$, the terms in \eqref{eq:cayley2} with $s_i$-degree $1$ correspond to $\wh{T}$ such that $i$ is a leaf, which by induction hypothesis sum to 
    \[
    N_{\mvs,t}\cdot s_{i}\cdot  \sum_{1\le j\le t+1,\  j\neq i} s_j.
    \]
    This accounts for all the terms since every term of \eqref{eq:indforest} has some $s_i$ of degree $1$, the latter following from the fact that every $\hat{T}$ has a leaf. Since forest $F$ is spanning, we have that $\sum_{i=1}^ts_i=n$ and hence~\eqref{eq:indforest} implies the desired result. 
\end{proof}

Notice that if $\pi$ is a uniformly chosen spanning tree on $K_n$ then
$
 \pr_0(e\in \pi)={2}/{n}.
$
If $\gC$ is a partial matching with $k$ edges then it is a spanning forest with $t=n-k$ trees that consists of $k$ trees of vertex-size $2$ and $n-2k$ trees of vertex-size $1$ thus $\pr_0(\gC\subseteq \pi)={2^k}/{n^{k}}.$
Thus, the left-hand side of~\eqref{eq:MSTPKnProd} can be treated as $\abs{{\pr_0(\gC\subseteq \pi)}/{p^k}-1}=0.$

\begin{obs}\label{obs:MSTP indep}
    Let $e_1,e_2,\ldots, e_k$ be a set of vertex-disjoint edges of a complete graph $K_n$, then $\{\1_{e_k\in \pi}\}_{i=1}^k$ is a family of mutually independent random variables under $\pr_0$.
\end{obs}
In particular, the presence of disjoint edges present in a uniformly spanning tree are independent of each other.
Finally, to confirm Assumption~\ref{as:non matching}, if $\gC\in\wt{\cT}\setminus\wt{\cM}$ with $k$ edges and $t$  connected components with sizes $\mvs=(s_1,s_2,\ldots,s_t)$, then
\[
\pr_0(\gC\subseteq \pi)=\frac{n^{n-k-2}}{n^{n-2}}\prod_{i=1}^t s_i\le \frac1{n^k}\left((s_1+s_2+\cdots+s_t)/t\right)^t = \frac1{n^k}(1+k/t)^t \le \frac{e^k}{n^{k}},
\]
where the first inequality is AM-GM inequality.
Thus, if $\{\go_e\}_{e\in E}$ are i.i.d.~and are such that \eqref{eq:psi} holds the MSTP on $K_{n}$ satisfies our assumptions and all  results from Section~\ref{sec mainres} apply. 

\begin{rem}[Minimal spanning forest with fixed number of trees]
    A similar analysis can be carried out to show that our assumptions are satisfied for an optimization problem over all spanning forests of $K_n$ with exactly $\ell$ trees for some fixed number $\ell$. 
\end{rem}

Theorems~\ref{thm:poi limit} yields that the size of the intersection of two independent samples from $\pr_{\gb,\mvgo}$ converges to $\poi\left(2e^{\psi(2\gb)-2\psi(\gb)}\right)$ in probability. As we mentioned in Remark~\ref{indepintersections}, if $\gb=0$, then the conclusion of Theorem~\ref{thm:poi limit} for MMP and TSP can be converted to a well-known result. Indeed, for these two problems, the intersection between two independent uniform samples can be rewritten in terms of the number of fixed points of the uniform permutation. Poisson convergence of the latter can be derived via several techniques, including the Stein--Chen method. On the other hand, in the case of MSTP, we are unaware of such a result, and it does not seem to be easily rephrased in terms of fixed points. Thus, we present direct proof of its limiting distribution via the Stein--Chen method.

\begin{thm}[Size of the intersection of two uniform spanning trees]\label{thm:MSTuniform}
    Suppose $T_1$ and $T_2$ are two spanning trees on $K_n$ chosen uniformly at random and independently. Let $W$ be the number of edges shared by $T_1$ and $T_2$,
    then 
    \[
    \dtv(W,S)= O(1/n),
    \]
    where $\dtv$ is the total variation distance and $S\sim\poi\,(2)$.
\end{thm}
\begin{proof}
    For any fixed edge $e\in E(K_n)$ and $T$ chosen uniformly from the set of all spanning trees, we have that 
    \[
    p_e:=\pr(e\in T)=\frac{n-1}{{n\choose 2}}=\frac{2}{n}.
    \]
    Since $T_1$ and $T_2$ are completely independent of each other, we have that
    \[
    \pr(e\in T_1\cap T_2)=p_e^2=\frac{4}{n^2}.
    \]
    This implies that
    \begin{align*}
    \E W&=\sum_{e\in E} p_e^2\to2\quad\text{as}\quad n\to\infty.
    \end{align*}
    
    Notice that if the edges $e$ and $e'$ do not share a vertex, then by Observation~\ref{obs:MSTP indep} we have that
    \[
    p_{e,e'}:=\pr(e\cup e'\in T)=\frac{n^{n-4}}{n^{n-2}}\cdot 2^2=\frac{4}{n^2}=p_ep_{e'}.
    \]
    Hence the events that disjoint edges are present in a spanning tree are pairwise independent.
    Similar application of Lemma~\ref{th:cayley} show that if $e$ and $e'$ shares a vertex then
    \[
    p_{e,e'}=\frac{n^{n-4}}{n^{n-2}}\cdot 3=\frac{3}{n^2}\le p_ep_{e'}.
    \]
    In particular,  using the fact that each edge $e$ shares a vertex with exactly $2(n-2)$ other edges we have
    \begin{align*} 
        \var(W) &=\sum_{e\in E} p_e^2 +\sum_{e,e'\in E\atop e\neq e'}p_{e,e'}^2- \left(\sum_{e\in E} p_e^2\right)^2\\
        &= \E W + (\E W)^2 + 2\cdot (n-2) \binom{n}{2}\cdot \left(\left(\frac{3}{n^2}\right)^2-\left(\frac{4}{n^2}\right)^2\right) - (\E W)^2\to 2 \text{ as } n\to\infty.
    \end{align*}
    Letting $S\sim \poi(2)$ the Stein--Chen method \cite[Theorem 7.1]{Dey_Poi_notes} (see also \cite[Theorem 1]{AGG90}) yields that
    \begin{align*}
        \dtv(W,S)
        &\le \frac{1}{2}\sum_{e\in E} \left(p_e^4+\sum_{e'\cap e\neq\emptyset} (p_e^2p_{e'}^2+p^2_{e,e'})\right)
        =\frac{1}{2}{n\choose 2} \left(\frac{16}{n^4}+2(n-2)\frac{25}{n^4}\right)= O\left(\frac1n\right).
    \end{align*}
    This completes the proof.
\end{proof} 


\subsection{Minimal $k$-factor problem on $K_{2n}$}\label{sec:kfact}
In terms of the setup from Section~\ref{intro setup} for the Minimal $k$-factor problem on $K_{2n}$ we have
\begin{itemize}
    \item $G_n=K_{2n}$, $\abs{V}=2n$, and $\abs{E}={2n\choose 2}$,
    \item $\cS$ - the set of all $k$-regular (not necessarily connected) spanning graphs on $G$, 
    \item $m=nk$, because each $k$-regular graph on $\abs{V}=2n$ vertices contains exactly $(2n\cdot k)/2$ edges,
    \item $p=k/(n-1)$, and
     $\gamma=k^2/4$.
\end{itemize}

Similarly to the applications above,  Assumptions~\ref{as:Graph} and~\ref{as:kappa} are clearly satisfied. Any partial matching on $K_{2n}$ can be extended to a $2n$-cycle, which can be extended to a $k$-regular subgraph, verifying Assumption~\ref{as:matching}. Hence it remains to check Assumptions~\ref{as:prod of edges} and~\ref{as:non matching}, \ie~if $\gC$ is a partial matching then
\begin{equation}\label{eq:kfactProd}
\abs{{\pr_0(\gC\subseteq \pi)}/{p^{\abs{\gC}}}-1}\le \frac{\rho^{\abs{\gC}}}{\sqrt{m}},    
\end{equation}
for some $\rho\in(0,\infty)$, and if $\gC\subseteq E$ then
\begin{equation}\label{eq:kfactProd2}
\pr_0(\gC\subseteq \pi)\le (\rho \cdot p)^{\abs{\gC}},
\end{equation}
for some $\rho\in(0,\infty)$.

To count the total number of $k$-regular subgraphs, we first construct a $k$-regular graph on $2n$ vertices with possible self-loops and multi-edges. This is commonly known as the configuration model. Consider each vertex with exactly $k$ half-edges and pair these half-edges among the vertices in all possible combinations. A simple computation yields that there $(2nk)!/\left(2^{nk}(nk)!\cdot k!^{2n}\right)$ such graphs, where the numerator counts all possible pairings of the half-edges and denominator takes care of double counting. Angel, van der Hofstad, and Holmgren~\cite{CMselfloops} established that the limiting distribution of self-loops and multi-edges is Poisson thus by conditioning the configuration model on not having self-loops and multi-edges, we get that there are 
\[
\frac{(2nk)!}{2^{nk}\cdot (nk)!\cdot (k!)^{2n}} \cdot e^{-\frac{k^2-1}{4}+O({k^2}/{n})}
\]
possible $k$-factors in $K_{2n}$.

Suppose $\gC\in\wt{\cM}$ then the number of $k$-factors that contain $\gC$ can be written as
\[
\frac{(2nk-2\abs{\gC})!}{2^{nk-\abs{\gC}}\cdot (nk-\abs{\gC})!\cdot (k!)^{2n-2\abs{\gC}}((k-1)!)^{2\abs{\gC}}} \cdot e^{-\frac{k^2-1}{4}+O({k^2}/{n})}
\]
Therefore,
\begin{align*}
    \abs{{\pr_0(\gC\subseteq \pi)}/{p^{\abs{\gC}}}-1}&\le\abs{\frac{(2n-1)^{\abs{\gC}}k^{\abs{\gC}} 2^{\abs{\gC}}\prod_{i=0}^{\abs{\gC}}(nk-i)}{\prod_{i=0}^{2\abs{\gC}-1}(2nk-i)}\cdot e^{O({k^2}/{n})}-1}\\
    &=\abs{\frac{(2n-1)^{\abs{\gC}}k^{\abs{\gC}}}{\prod_{i=0}^{\abs{\gC}-1}(2nk-2i-1)}\cdot e^{O({k^2}/{n})}-1}\\
    &\le \abs{\prod_{i=0}^{\abs{\gC}-1} \left(1-\frac{2i+1-k}{2nk-k}\right)^{-1}\cdot e^{O({k^2}/{n})}-1}\le \frac{\rho^{\abs{\gC}}}{n}.
\end{align*}
for some $\rho \in(0,\infty)$  and so Assumption~\ref{as:prod of edges} is verified.
Finally, suppose that $\gC\notin \wt{\cM}$, then the number of $k$-factors that contain $\gC$ can be written as
\[
\frac{(2nk-2\abs{\gC})!}{2^{nk-\abs{\gC}}\cdot (nk-\abs{\gC})!\cdot \prod_{i\in \gC} (k-\deg_{\gC}(i))!\prod_{i\notin \gC} k!} \cdot e^{-\frac{k^2-1}{4}+O({k^2}/{n})},
\]
where $\deg_{\gC}(i)$ represents the degree of vertex $i$ in $\gC$. Hence, denoting the falling factorial by $(k)_\ell:=k!/(k-\ell)!$ we can upper bound the ratio of interest as
\begin{align*}
    \frac{\pr_0(\gC\subseteq \pi)}{p^{\abs{\gC}}}&=\frac{(2n-1)^{\abs{\gC}}\prod_{i\in\gC} (k)_{\deg_{\gC}(i)}}{k^{\abs{\gC}}\cdot\prod_{i=0}^{\abs{\gC}-1}(2nk-2i-1)}\cdot e^{O(k^2/n)}\\
    &\le\frac{(2n-1)^{\abs{\gC}}k^{\abs{\gC}}}{\prod_{i=0}^{\abs{\gC}-1}(2nk-2i-1)}\cdot e^{O(k^2/n)}\le \rho^{\abs{\gC}}
\end{align*}
for some $\rho \in(0,\infty)$ and so Assumption~\ref{as:non matching} is verified as well.
Thus, if $\{\go_e\}_{e\in E}$ are i.i.d.~and are such that \eqref{eq:psi} holds the minimal $k$-factor problem on $K_{2n}$ satisfies our assumptions and all of the results form Section~\ref{sec mainres} apply.

\section{Proofs of main results}\label{sec proofs}

Although the proof of our results may seem technical, the key observation is rather simple. Most bounds rely on the fact that the probability that a set of selected edges belongs to a uniformly chosen configuration is maximized when this set is a partial matching, \ie~it consists of vertex-disjoint edges.

\subsection{Proof of CLT for the log-partition function}

Here, we present the proof of Lemma~\ref{lem:Zhat} and Theorem~\ref{thm:log part clt}, the latter follows directly from the cluster decomposition of $\wh{Z}$ given by Lemma~\ref{lem:Zhat}.

\begin{proof}[Proof of Lemma~\ref{lem:Zhat}]
Recall that, as defined in~\eqref{eq:p}, the probability that an edge $e\in E$ belongs to a uniformly chosen configuration $\pi\in \cS$ is equal to  
\begin{equation*}
    p=p_n(e):=\pr_0(e\in\pi)=\frac{\abs{\pi}}{E}=\frac{m}{E}.
\end{equation*}

Let $\wt{\cS}$ be the set of all subgraphs of elements of $\cS$ and recall that we denoted $\xi_\gC:=\prod_{e\in \gC}\xi_e.$  First, we rewrite $\wh{Z}(\gb)$ as
\begin{align*}
    \wh{Z}(\gb)&=\E_0 \prod_{e\in E} (1+\xi_e\1_{e\in\pi})
    = \sum_{\gC\subseteq E} \pr_0(\gC\subseteq\pi)\xi_\gC
    = \sum_{\gC\in\wt{\cS}}\pr_0(\gC\subseteq\pi)\xi_\gC.
\end{align*}
Here, in the third equality, we passed from summing over $\gC\subseteq E$ to $\gC\in \wt{\cS}$ because for the remaining $\gC$, the probability of it being contained in $\pi$ is zero.
On the other hand
$
\prod_{e\in E}\left(1+p\cdot\xi_e\right)
=\sum_{\gC\subseteq E}p^{\abs{\gC}}\xi_\gC.
$
Thus
\begin{align}\label{eq:dif gen}
    \prod_{e\in E}\left(1+p\cdot\xi_e\right) - \wh{Z}_n(\gb)
    &=\sum_{\gC\notin\wt{\cS}}p^{\abs{\gC}}\xi_\gC+\sum_{\gC\in\wt{\cS}}\left(p^{\abs{\gC}}-\pr(\gC\subseteq\pi)\right)\xi_\gC.
\end{align}

The first term of the right-hand side of~\eqref{eq:dif gen} is negligible because it has mean zero, and its variance goes to $0$. 
Indeed, since the sum is over $\gC\notin\wt{\cS}$, \ie~subgraphs $\gC$ of $G$ that cannot be contained in any configuration $\pi\in\cS$, we can upper bound it by the sum over all $\gC\notin\wt{\cM}$. Here we relied on Assumption~\ref{as:matching} which implies that $\wt{\cM}\subseteq \wt{\cS}$. Note that the number of $\gC\notin\wt{\cM}$ with $k$ edges is at most $\binom{E}{k-1}2(k-1)|V|$. Here we used the fact that one can construct such a $\gC$ by selecting $(k-1)$ edges from $E$ arbitrarily in $\binom{E}{k-1}$ ways, select one vertex from the chosen set of edges in $2(k-1)$ ways, and then adding one adjacent edge to ensure that it is not a partial matching.
Thus,
\begin{align}
\begin{split}\label{eq:dif gen not subseteq}
         \E_0\left(\sum_{\gC\notin\wt{\cS}}p^{\abs{\gC}}\xi_\gC\right)^2
     = \sum_{\gC\notin\wt{\cS}}(p\cdot v_\gb)^{2\abs{\gC}} 
     &\le  \sum_{\gC\notin\wt{\cM}}(p\cdot v_\gb)^{2\abs{\gC}}  \\
     &\le\sum_{k=2}^E {E \choose {k-1}}2(k-1)|V| \left(p\cdot v_\gb\right)^{2k} \\
     &\le\frac{2|V|}{E}\sum_{k=2}^E \frac{ (v_\gb^2\cdot m^2/E)^k }{ (k-2)!}=O(1/m),
     \end{split}
\end{align}
where in the last step we used that $p=m/E=\Theta(E^{-1/2})$.
The second term on the right-hand side of~\eqref{eq:dif gen} can be controlled in a similar fashion as it also has mean zero, and the variance approaches $0$. We first split the sum into two cases, one with the sum over partial matchings $\gC\in\wt{\cM}$ and the second with over the rest of the elements in $\wt{\cS}$.
Assumption~\ref{as:prod of edges} yields that the contribution of partial matchings is negligible. Indeed for some $\rho\in(0,\infty)$ we have that
\begin{align}
\begin{split}\label{eq:dif gen subseteq}
    \E_0\left(\sum_{\gC\in\wt{\cM}}\left(p^{\abs{\gC}}-\pr_0(\gC\subseteq\pi)\right)\xi_\gC\right)^2&\le \sum_{k=1}^m\sum_{\gC\in\wt{\cM}\atop \abs{\gC}=k}\abs{1-\frac{\pr_0(\gC\subseteq\pi)}{p^k}}^2 (p\cdot v_\gb)^{2k}\\
    &\le \sum_{k=1}^m{E \choose k}\frac{\rho^{2k} (p\cdot v_\gb)^{2k}}{m}\\
    &\lesssim \sum_{k=1}^m\frac{m^{2k}}{k!}\frac{(\rho\cdot v_\gb)^{2k}}{m^{2k+1}}\lesssim\frac{1}{m}. 
    \end{split}
\end{align}

On the other hand, the remaining term can be treated as follows
\begin{align}
 \E_0\biggl(\sum_{\gC\in\wt{\cS}\setminus\wt{\cM}}\left(p^{\abs{\gC}}-\pr_0(\gC\subseteq\pi)\right)\xi_\gC\biggr)^2
 &\le2 \sum_{\gC\notin\wt{\cM}}p^{2\abs{\gC}}v_\gb^{2\abs{\gC}}\label{eq:errwhZ1}\\
 &\quad+2\sum_{\gC\in\wt{\cS}\setminus\wt{\cM}}\pr_0(\gC\subseteq\pi)^2\cdot v_\gb^{2\abs{\gC}}\label{eq:errwhZ2}.
\end{align}
Clearly, by the same computations as in~\eqref{eq:dif gen not subseteq}, the term in the right-hand side of~\eqref{eq:errwhZ1} is at most of order $1/m$. Finally, by Assumption~\ref{as:non matching}, the following inequality holds for some $\rho\in (0,\infty)$
\begin{align}
\eqref{eq:errwhZ2}&\le  \sum_{\gC\notin\wt{\cM}} (\rho\cdot p^2\cdot v_\gb^2)^{\abs{\gC}}.
\end{align}
Similar computations to those in~\eqref{eq:dif gen not subseteq} now yield that~\eqref{eq:errwhZ2} is also at most of order $1/m$. This concludes the proof.
\end{proof}

\begin{proof}[Proof of Theorem~\ref{thm:log part clt}]
    By Lemma~\ref{lem:Zhat} we have that
\begin{align*}
\log\wh{Z}(\gb)
    &=\sum_{e\in E}\log\left(1+p\xi_e\right)+\op(1)
    =p\sum_{e\in E}\xi_e-\frac12p^2 \sum_{e\in E}\xi_e^2+R_n+\op(1),
\end{align*}
where $R_n$ tends to $0$ as $n\to\infty$ due to $\xi_e$ has finite $2$-nd moment (Assumption~\ref{as:weights}).
By~\eqref{eq:gc},~\eqref{eq:p}, and the fact that $\{\xi_e\}_{e\in E}$ are i.i.d.~we have that 
$$
p\sum_{e\in E}\xi_e=\frac{m}{E}\sum_{e\in E}\xi_e\Rightarrow \N\left(0,2\gc  v_\gb^2\right).$$
Similarly, the SLLN yields that 
\[
\frac12p^2\sum_{e\in E}\xi_e^2=\frac{m^2}{2E}\cdot \frac{1}{E}\sum_{e\in E}\xi_e^2\to \gc v_\gb^2\quad \text{a.s.}
\]
 Hence, we have that
$\log\wh{Z}_n(\gb)\Rightarrow \N\left(-\gc v_\gb^2,2\gc  v_\gb^2\right).
$
\end{proof}

\subsection{Proof of the Poisson limiting distribution of the intersection}

We now present the proof of Theorem~\ref{thm:poi limit}, using cluster expansion to establish the convergence in probability of the moment generating function to that of the $\poi\left(2\gc\cdot e^{\psi(2\gb)-2\psi(\gb)}\right)$ random variable.

Let us start with an overview of the proof of Theorem~\ref{thm:poi limit}. We prove that the moment generating function (mgf) of the size of the intersection $\abs{\pi\cap\pi'}$ converges to that of the desired Poisson random variable. Here, we recall that for a Poisson($\gl$) r.v.~the mgf is $\exp(\gl(e^{\theta}-1)), \theta\in\dR$. We first write the mgf of the size of the intersection $\abs{\pi\cap\pi'}$ under the Gibbs measure as a ratio of two random variables 
$$
\E_0\exp\bigl(\theta\abs{\pi\cap\pi'}+\sum_{e\in E}\go_e\cdot (\ind_{e\in\pi}+\ind_{e\in\pi'})\bigr) \qquad\text{ and }\qquad\wh{Z}^2(\gb),
$$
and analyze them separately. For $\wh{Z}(\gb)$, the analysis was done in Lemma~\ref{lem:Zhat}. We employ a similar but more involved analysis for the term $\E_0\exp\left(\theta\abs{\pi\cap\pi'}+\sum_{e\in E}\go_e\cdot (\ind_{e\in\pi}+\ind_{e\in\pi'})\right)$, as it involves an average with respect to two uniformly chosen elements $\pi,\pi'$ from $\E_0$.

\begin{proof}[Proof of Theorem~\ref{thm:poi limit}]
    Recall that, $\gibbs{\cdot}_\beta$ denotes the Gibbs average. Consider the moment generating function $\gibbs{\exp(\theta\abs{\pi\cap \pi'})}_\gb$ and rewrite it as
    \begin{equation}\label{eq:MGF MMP gen}
        \gibbs{\exp(\theta\abs{\pi\cap \pi'})}_\gb=\frac{1}{\wh{Z}^2(\gb)}\E_0\exp\bigl(\theta\abs{\pi\cap\pi'}+\sum_{e\in E}\go_e\cdot (\ind_{e\in\pi}+\ind_{e\in\pi'})\bigr).
    \end{equation}
    Here $\E_0$ denotes expectation with respect to two independent samples $\pi,\pi'$ uniformly chosen from $\cS$.
    The denominator of~\eqref{eq:MGF MMP gen} can be written using Lemma~\ref{lem:Zhat} as 
    \begin{align}\label{eq:MGF denom MMP gen}
        \wh{Z}(\gb) = \prod_{e\in E}\left(1+p\cdot\xi_e\right)\cdot\left(1+\op\left(1\right)\right),
    \end{align}
    and the numerator of~\eqref{eq:MGF MMP gen} can be written as 
    \begin{align}
    \begin{split}\label{eq:MGF numer MMP gen expansion1}
        &\E_0\exp\bigl(\theta\abs{\pi\cap\pi'}+\sum_{e\in E}\go_e\cdot (\ind_{e\in\pi}+\ind_{e\in\pi'})\bigr)\\
        &\qquad\qquad=\E_0\prod_{e\in E}\left(1+\xi_e\1_{e\in\pi}\right)\bigl(1+\xi_e\1_{e\in\pi'}\bigr)\bigl(1+(e^\theta-1)\1_{e\in\pi\cap\pi'}\bigr)\\
        &\qquad\qquad=\E_0\prod_{e\in E}\biggl(1+\xi_e\cdot(\1_{e\in\pi}+\1_{e\in\pi'})+\bigl((e^\theta-1)(1+\xi_e)^2+\xi_e^2\bigr)\cdot \1_{e\in\pi\cap\pi'}\biggr).
            \end{split}
    \end{align}

        Letting $Y_e:=(e^\theta-1)(1+\xi_e)^2+\xi_e^2$ and using the cluster expansion in a similar way as in the proof of Lemma~\ref{lem:Zhat}, we rewrite
    \begin{align}
    \begin{split}\label{eq:MGF numer MMP gen expansion2}
        ~\eqref{eq:MGF numer MMP gen expansion1}
        &=\sum_{\gC_1,\gC_2,\gC_3\subseteq  E\text{ disj.}} \E_0\bigl(\prod_{e\in \gC_1} \xi_e\1_{e\in\pi}\prod_{e\in \gC_3} \1_{e\in\pi}\bigr) \cdot \E_0\bigl(\prod_{e\in \gC_2} \xi_e\1_{e\in\pi'}\prod_{e\in \gC_3} \1_{e\in\pi'}\bigr) \cdot \prod_{e\in \gC_3} Y_e\\
        &=\sum_{\gC_1,\gC_2,\gC_3\in \wt{\cS}\text{ disj.}} \pr_0(\gC_1\cup\gC_3\subseteq \pi)\pr_0(\gC_2\cup\gC_3\subseteq \pi)\prod_{e\in \gC_1\cup \gC_2}\xi_e\prod_{e\in \gC_3}Y_e.
        \end{split}
    \end{align}
    In the last equality of~\eqref{eq:MGF numer MMP gen expansion2}, we used the fact that for a subset of $E$ not in $\wt{\cS}$, the probability being contained in $\pi$ is zero. Moreover, by disj.~we mean disjoint as a subset of the edgeset $E$.
    On the other hand, we have that
    \begin{align}
    \begin{split}\label{eq:MGF numer MMP gen expansion3}
        \prod_{e\in E}\left(1+2p\cdot\xi_e+p^2\cdot Y_e\right)
        &=\prod_{e\in E}\left(1+p\cdot\xi_e+ p\cdot\xi_e+p^2\cdot Y_e\right)\\
        &=\sum_{\gC_1,\gC_2,\gC_3\subseteq E \text{ disj.}}p^{\abs{\gC_1}+\abs{\gC_2}+2\abs{\gC_3}}\prod_{e\in\gC_1\cup\gC_2}\xi_e\prod_{e\in\gC_3}Y_e.
        \end{split}
    \end{align}
 
    Next, we claim that the difference of~\eqref{eq:MGF numer MMP gen expansion2} and~\eqref{eq:MGF numer MMP gen expansion3} is negligible. However the analysis is slightly different from that in the proof of Lemma~\ref{lem:Zhat} as random variables $Y_e$ are not mean zero. Firstly, we still rewrite the difference as two separate terms based on the validity of the clusters
    \begin{equation}\label{eq:MGF numer MMP gen diff}
       \abs{\eqref{eq:MGF numer MMP gen expansion3}-\eqref{eq:MGF numer MMP gen expansion2}}
       =\abs{\textrm{Err}_1-\textrm{Err}_2}
       \le \abs{\textrm{Err}_1} + \abs{\textrm{Err}_2},
    \end{equation}
    where
    \begin{align*}
        \textrm{Err}_1 &:=\sum_{\gC_1\cup\gC_2\cup\gC_3\notin \wt{\cS}\text{ disj.}}p^{\abs{\gC_1}+\abs{\gC_2}+2\abs{\gC_3}} \cdot \prod_{e\in \gC_1\cup\gC_2}\xi_e\prod_{e\in \gC_3}Y_e, \text{ and }\\
        \textrm{Err}_2 &:=\sum_{\gC_1\cup\gC_2\cup\gC_3\in\wt{\cS} \text{ disj.}} \left(\pr_0(\gC_1\cup\gC_3\subseteq \pi)\pr_0(\gC_2\cup\gC_3\subseteq \pi)-p^{\abs{\gC_1}+\abs{\gC_2}+2\abs{\gC_3}}\right)\cdot\prod_{e\in \gC_1\cup\gC_2}\xi_e\prod_{e\in \gC_3}Y_e,
    \end{align*}
     where both of these sums are over triples $(\gC_1,\gC_2,\gC_3)$ satisfying the corresponding conditions.
     Next, we bound $L_1$ norms of these terms separately.\medskip
    
    \noindent\textbf{Upper bound on $\norm{\textrm{Err}_1}_1$.}
    We bound the first error term $\textrm{Err}_1$ from~\eqref{eq:MGF numer MMP gen diff} by dividing the sum further into two sums: one over $\gC_3\in \wt{\cS}$ and the other over $\gC_3\notin\wt{\cS}$.
    \begin{align}
    \begin{split}\label{eq:Err1}
    \norm{\textrm{Err}_1}_1
    &\le\sum_{\gC_3\in\wt{\cS}}p^{2\abs{\gC_3}}\E_0\biggl( \prod_{e\in \gC_3}Y_e \cdot \biggl|\sum_{{\gC_1,\gC_2:\atop\gC_1\cup\gC_2\cup \gC_3\notin\wt{\cS} \text{ disj.}}}p^{\abs{\gC_1}+\abs{\gC_2}}\cdot  \prod_{e\in \gC_1\cup\gC_2}\xi_e\biggr|\biggr)\\
    &\qquad+\sum_{\gC_3\notin\wt{\cS}}p^{2\abs{\gC_3}}\E_0\biggl( \prod_{e\in \gC_3}Y_e \cdot \biggl|\sum_{{\gC_1\cup\gC_2\subseteq E\atop \text{disj.~w/ each other and } \gC_3}}p^{\abs{\gC_1}+\abs{\gC_2}}\cdot  \prod_{e\in \gC_1\cup\gC_2}\xi_e\biggr|\biggr).
        \end{split}
    \end{align}
    Applying Cauchy--Schwarz inequality for both terms, we get
    \begin{align}
    \eqref{eq:Err1}
    &\le\sum_{\gC_3\in\wt{\cS}}p^{2\abs{\gC_3}}\sqrt{(\E Y_e^2)^{\abs{\gC_3}}}\cdot
    \sqrt{\sum_{\gC_1,\gC_2\,:\,\gC_1\cup\gC_2\cup \gC_3\notin\wt{\cS}\text{  disj.}}(p\cdot  v_\gb)^{2\abs{\gC_1\cup\gC_2}}}\label{eq:Err21}\\
    &\quad+\sum_{\gC_3\notin\wt{\cS}}p^{2\abs{\gC_3}}\sqrt{(\E Y_e^2)^{\abs{\gC_3}}}\cdot\sqrt{\sum_{\gC_1\cup\gC_2\subseteq E \text{ disj.~w/ each other and } \gC_3}(p\cdot  v_\gb)^{2\abs{\gC_1\cup\gC_2}}} \,\,.\label{eq:Err22}
    \end{align}

    Recall that in the analogous step in the proof of Lemma~\ref{lem:Zhat}, namely~\eqref{eq:dif gen not subseteq}, we bounded the sum over the clusters $\gC\notin \wt{\cS}$ by the sum over all clusters $\gC\notin\wt{\cM}$. Following the same idea here, we upper bound the orders of both sums inside of the square roots of~\eqref{eq:Err21} and~\eqref{eq:Err22}. For~\eqref{eq:Err21}, we bound the sum inside the square root by $1/m$ and the outside sum by a constant. For~\eqref{eq:Err22}, we bound the sum inside the square root by a constant and the outside sum by $1/m$. Combining we get that
    \begin{equation}\label{eq:MGF numer MMP gen Err1 L1}
        \norm{\textrm{Err}_1}_1\lesssim \frac{1}{\sqrt{m}}.
    \end{equation}
    
    \noindent\textbf{Upper bound on $\norm{\textrm{Err}_2}_1$.}
    Now in the second error term $\textrm{Err}_2$ the sum over clusters $\gC_1,\gC_2,\gC_3\in\wt{\cS}$, subject to constraint $\gC_1\cup\gC_2\cup\gC_3\in\wt{\cS}$. We again consider two cases:
    \begin{enumerate}
        \item[Case 1.] the union of the clusters $\gC_1\cup\gC_2\cup\gC_3$ is a matching,
        \item[Case 2.] the union of the clusters $\gC_1\cup\gC_2\cup\gC_3$ is not a matching.
    \end{enumerate}

    In Case 1., the difference term inside the sum in $\textrm{Err}_2$, for simplicity denoted by $D(\gC_1,\gC_2,\gC_3)$, can be bounded as follows
    \begin{align*}
        \abs{D(\gC_1,\gC_2,\gC_3)}&=\abs{\pr_0(\gC_1\cup\gC_3\subseteq \pi)\pr_0(\gC_2\cup\gC_3\subseteq \pi)-p^{\abs{\gC_1}+\abs{\gC_2}+2\abs{\gC_3}}}\\
        &\le p^{\abs{\gC_1}+\abs{\gC_2}+2\abs{\gC_3}}\cdot\frac{1}{\sqrt{m}}\cdot \bigl(\rho^{\abs{\gC_1\cup\gC_3}}+\rho^{\abs{\gC_2\cup\gC_3}}\bigr),
    \end{align*}
    where in the last inequality we used Assumption~\ref{as:prod of edges}. Furthermore, fixing $\gC_3\in \wt{\cM}$ and using similar computations to those in \eqref{eq:dif gen subseteq} we have that
    \begin{align}
    &\sum_{\gC_1,\gC_2\in\wt{\cM}\atop \text{disj.~w/ each other and } \gC_3}
    D(\gC_1,\gC_2,\gC_3)^2 \cdot  v_\gb^{2\abs{\gC_1\cup\gC_2}}\notag\\
     &\qquad\qquad\le \sum_{{\gC_1,\gC_2\in\wt{\cM}\atop \text{ disj.~w/each other and } \gC_3}} p^{2\abs{\gC_1}+2\abs{\gC_2}+4\abs{\gC_3}}\cdot\frac{\left(\rho^{\abs{\gC_1\cup\gC_3}}+\rho^{\abs{\gC_2\cup\gC_3}}\right)^2}{m}\cdot  v_\gb^{2\abs{\gC_1\cup\gC_2}}\notag\\
     &\qquad\qquad=p^{4\abs{\gC_3}}\sum_{k_1=1}^{E-\abs{\gC_3}}\sum_{k_2=1}^{E-\abs{\gC_3}-k_1}{E-\abs{\gC_3} \choose k_1,\,k_2}\cdot\frac{\left(\rho^{k_1}+\rho^{k_2}\right)^2\rho^{2\abs{\gC_3}}}{m}\cdot  (p\cdot v_\gb)^{2(k_1+k_2)}\notag\\
     &\qquad\qquad\lesssim\frac{\rho^{2\abs{\gC_3}}}{m}p^{4\abs{\gC_3}}.\label{eq:bound on D 1}
    \end{align}
    
    In Case 2., at least one of $\gC_1,\gC_2,$ and $\gC_3$ is not a matching; we use the same idea as in~\eqref{eq:errwhZ1} from the proof of Lemma~\ref{lem:Zhat}. We separate the difference into two further subterms
    \begin{align}\label{eq:bound on D 2}
         \abs{D(\gC_1,\gC_2,\gC_3)}
         &\le\pr_0(\gC_1\cup\gC_3\subseteq \pi)\pr_0(\gC_2\cup\gC_3\subseteq \pi)+p^{\abs{\gC_1}+\abs{\gC_2}+2\abs{\gC_3}}.
    \end{align}

    It remains to note that the following inequality holds, by analogous reasons to those from the proof of Lemma~\ref{lem:Zhat}, 
    \[
    \biggl\Vert\sum_{{\gC_1,\gC_2,\gC_3\in\wt{\cS}\atop {\gC_1\cup\gC_2\cup\gC_3\notin \wt{\cM} \text{ disj.}}}} D(\gC_1,\gC_2,\gC_3)\cdot \prod_{e\in \gC_1\cup\gC_2}\xi_e\prod_{e\in \gC_3}Y_e\biggr\Vert_1
    \lesssim \frac{1}{\sqrt{m}}.
    \]
    Next we bound $\textrm{Err}_2$ by decomposing it in a similar way to \eqref{eq:Err1}
    \begin{align*}\label{eq:MGF numer MMP gen Err1 L2}
    \norm{\textrm{Err}_2}_1
    &\le \biggl\Vert\sum_{\gC_3\in\wt{\cS}}\E_0\biggl( \prod_{e\in \gC_3}Y_e \cdot \bigl|\sum_{{\gC_1,\gC_2\in\wt{\cS}\atop \gC_1\cup\gC_2\cup\gC_3\in \wt{\cM} \text{ disj.} }} D(\gC_1,\gC_2,\gC_3)\cdot  \prod_{e\in \gC_1\cup\gC_2}\xi_e\bigr|\biggr)\biggr\Vert_1\notag\\
    &\qquad\qquad+\biggl\Vert\sum_{{\gC_1,\gC_2,\gC_3\in\wt{\cS}\atop \gC_1\cup\gC_2\cup\gC_3\notin \wt{\cM} \text{ disj.}}} D(\gC_1,\gC_2,\gC_3)\cdot \prod_{e\in \gC_1\cup\gC_2}\xi_e\prod_{e\in \gC_3}Y_e\biggr\Vert_1\notag.
        \end{align*}
        
Now, letting $u^2_\gb:=\E Y_e^2$ using Cauchy--Schwarz inequality and the bounds in~\eqref{eq:bound on D 1} and~\eqref{eq:bound on D 2}, we get

    \begin{align*}
        \norm{\textrm{Err}_2}_1&\le \frac{1}{\sqrt{m}}+\sum_{\gC_3\in\wt{\cS}}\sqrt{(\E Y_e^2)^{\abs{\gC_3}}}\cdot\sqrt{\sum_{\gC_1,\gC_2\in\wt{\cS}\atop \gC_1\cup\gC_2\cup\gC_3\in \wt{\cM} \text{ disj.} }D(\gC_1,\gC_2,\gC_3)^2\cdot  v_\gb^{2\abs{\gC_1\cup\gC_2}}}\notag\\
       &\lesssim \frac{1}{\sqrt{m}}+\sum_{\gC_3\in\wt{\cS}}u_\gb^{\abs{\gC_3}}\cdot\sqrt{\frac{\rho^{2\abs{\gC_3}}}{m}p^{4\abs{\gC_3}}}
        \le\frac{1}{\sqrt{m}}+\sum_{k_3=1}^E{E\choose k_3}(p^{2}\cdot u_\gb\cdot\rho)^{k_3}\cdot\frac{1}{\sqrt{m}}\lesssim \frac{1}{\sqrt{m}}.
    \end{align*}
Thus
\[
    \norm{\textrm{Err}_1}_1+\norm{\textrm{Err}_2}_1 =o(1).
\]
Combining~\eqref{eq:MGF numer MMP gen diff},~\eqref{eq:MGF numer MMP gen Err1 L1} and~\eqref{eq:MGF numer MMP gen Err1 L1} we get that
\[
\norm{\E_0\exp\bigl(\theta\abs{\pi\cap\pi'}+\sum_{e\in E}\go_e\cdot (\ind_{e\in\pi}+\ind_{e\in\pi'})\bigr)-\prod_{e\in E}\left(1+2p\cdot\xi_e+p^2\cdot Y_e\right)}_1=o(1),
\]
where $Y_e=(e^\theta-1)(1+\xi_e)^2+\xi_e^2$. 
Note that
\[
1+2p\cdot\xi_e+p^2\cdot Y_e = (1+p\cdot\xi_e)^2\cdot \biggl(1+(e^\theta-1)\left(\frac{1+\xi_e}{1+p\cdot\xi_e}\right)^2p^2\biggr).
\]
Therefore, by~\eqref{eq:MGF numer MMP gen diff}, we get
    \begin{align}\label{eq:MGF numer MMP gen exp form}
        \log &\E_0\exp\biggl(\theta\abs{\pi\cap\pi'} +\sum_{e\in E}\go_e\cdot (\ind_{e\in\pi}+\ind_{e\in\pi'})\biggr)\notag\\
        &=\log\prod_{e\in E}\left(1+p\cdot\xi_e\right)^2+\sum_{e\in E}\log\biggl(1+(e^\theta-1)\biggl(\frac{1+\xi_e}{1+p\cdot\xi_e}\biggr)^2p^2\biggr)+\log(1+\op\left(1\right))\\
        &=\log\prod_{e\in E}\left(1+p\cdot\xi_e\right)^2+\sum_{e\in E}(e^\theta-1)\left(\frac{1+\xi_e}{1+p\cdot\xi_e}\right)^2p^2+\op\left(1\right),\notag
    \end{align}    
    where in the last step we use the fact that $\log(1+x)=x+o(|x|^{1+\eps})$ for any $\eps>0$ and $\op\left(1\right)$ is justified by the finiteness of the $2(1+\eps)$--th moment of $\xi_e$ (Assumption~\ref{as:weights}). Now, by the SLLN, we have that
    \begin{align*}
        &\sum_{e\in E}\left(\frac{1+\xi_e}{1+p\cdot\xi_e}\right)^2p^2=2\frac{m^2}{2E}\cdot\frac{1}{E}\sum_{e\in E}\left(\frac{1+\xi_e}{1+p\cdot\xi_e}\right)^2\to 2\gamma\E(1+\xi_e)^2 
    \end{align*}
    almost surely.
    Recalling the definition of $\xi_e$, as in~\eqref{eq:v}, we derive
    \begin{equation}\label{eq:MGF numer MMP gen slln}
    2\gc\E(1+\xi_e)^2 =2\gc e^{\psi(2\gb)-2\psi(\gb)}.
    \end{equation}

    Finally, coming back to the original quantity of interest, we conclude that
    \begin{align}\label{eq:log MGF}
        \log \gibbs{e^{\theta\abs{\pi\cap\pi'}}}_\gb&=\log\E_0\exp\left(\theta\abs{\pi\cap\pi'}+\sum_{e\in E}\go_e\cdot (\ind_{e\in\pi}+\ind_{e\in\pi'})\right)-\log\wh{Z}^2(\gb)\notag\\
       \textnormal{[by~\eqref{eq:MGF denom MMP gen},~\eqref{eq:MGF numer MMP gen exp form} and~\eqref{eq:MGF numer MMP gen slln}] } &=\log\prod_{e\in E}\left(1+p\cdot\xi_e\right)^2\exp\left(2\gc(e^\theta-1)e^{\psi(2\gb)-2\psi(\gb)}\right)\notag\\
       &\qquad\qquad\qquad-\log\prod_{e\in E}\left(1+p\cdot\xi_e\right)^2+\op(1)\notag\\
       &=2\gc(e^\theta-1) e^{\psi(2\gb)-2\psi(\gb)}+\op(1).
    \end{align}
Thus,
\begin{align}\label{eq:conv}
    \gibbs{\exp(\theta\abs{\pi\cap \pi'})}_\gb \longrightarrow \exp\biggl(2\gc(e^\theta-1)e^{\psi(2\gb)-2\psi(\gb)}\biggr)
\end{align}
in probability w.r.t.~$\pr_{\beta,\mvgo}$. Finally, we note that the r.h.s.~of equation~\eqref{eq:conv} is the moment generating function of Poisson$(2\gc  e^{\psi(2\gb)-2\psi(\gb)})$ distribution. This completes the proof.
\end{proof}

\subsection{Proof of the quenched CLT for a typical configuration}

We prove Theorem~\ref{thm:typical CLT} by the method of moment-generating functions. First, we state Lemma~\ref{lem: prod ratio}, which provides the main technical estimate for the proof.

\begin{lem}\label{lem: prod ratio}
    Under Assumptions~\ref{as:Graph}--\ref{as:optimization} for any $\gb\in [0,\infty)$ and $\theta\in \bR$ we have that
    \[
    \prod_{e\in E}\frac{1+p\xi_e\left(\gb+\theta/\sqrt{m} \right)}{1+p\xi_e(\gb)} \longrightarrow 1
    \]
    in probability.
\end{lem}

\begin{proof}
    By taking the logarithm of the product, we get that
    \begin{align*}
        \log\prod_{e\in E}\frac{1+p\cdot\xi_e\left(\gb+{\theta}/{\sqrt{m}}\right)}{1+p\cdot \xi_e(\gb)}
        &=\sum_{e\in E}\left[\log\left( 1+p\cdot\xi_e\left(\gb+{\theta}/{\sqrt{m}}\right)\right)-\log\left(1+p\cdot\xi_e(\gb)\right)\right]\notag\\
        &\le \frac{m}{E-m}\cdot\sum_{e\in E}\abs{\xi_e\left(\gb+{\theta}/{\sqrt{m}}\right) -\xi_e(\gb)}.
    \end{align*}
    The last inequality follows from the fact that $\abs{\log(1+x)-\log(1+y)}\le \abs{x-y}/(1+\min(x,y))$ and $\xi_e\ge -1$ almost surely.
    Taking the $L^2$ norm and using independence, we get that
    \begin{align}
        \norm{\log\prod_{e\in E}\frac{1+p\cdot\xi_e\left(\gb+{\theta}/{\sqrt{m}}\right)}{1+p\cdot \xi_e(\gb)}}_2
        \lesssim\norm{\xi_e\left(\gb+\theta/\sqrt{m}\right) -\xi_e(\gb)}_2.
    \end{align}
    Recall that $\xi_e=e^{-\gb\go_e-\psi(\gb)}-1$. Thus, Taylor expansion and DCT yield that
    \begin{align*}
        &\norm{\xi_e\left(\gb+\theta/\sqrt{m}\right) -\xi_e(\gb)}_2^2\\
        &=\E\left(e^{-\gb\go_e-\psi(\gb)}
        \left(e^{-\theta \go_e/\sqrt{m}-\psi(\gb+\theta/\sqrt{m})+\psi(\gb)}-1\right)\right)^2=o(1).
    \end{align*}
    This completes the proof.
\end{proof}

\begin{proof}[Proof of Theorem~\ref{thm:typical CLT}]
    Consider the moment-generating function 
    \begin{align}\label{eq:char func}
        \gibbs{\exp\left(\frac{\theta}{\sqrt{m}}\left(W(\pi)+m\psi'(\gb)\right)\right)}_{\gb,\go}&=e^{-\frac{\theta}{\sqrt{m}}m\psi'(\gb)}\cdot\frac{Z\left(\gb+\theta/\sqrt{m} \,,\,\go\right)}{Z(\gb,\go)},
    \end{align}
    for $\theta\in\dR$.
    By Lemma~\ref{lem:Zhat} we know that  
    \[
    Z(\gb,\go)=e^{m\psi(\gb)}\wh{Z}(\gb,\go)=e^{m\psi(\gb)}\left(\prod_{e\in E}1+p\cdot\xi_e\right)(1+\op(1)).
    \]
     Using this representation of the partition function we get
    \begin{align}
       ~\eqref{eq:char func}=e^{m\left(\psi\left(\gb+\frac{\theta}{\sqrt{m}}  \right)-\psi(\gb)-\frac{\theta}{\sqrt{m}} \psi'(\gb)\right)}\cdot \frac{\prod_{e\in E}\left(1+p\cdot\xi_e\left(\gb+\theta/\sqrt{m} \right)\right)(1+\op(1))}{\prod_{e\in E}\left(1+p\cdot\xi_e\right)(1+\op(1))}.\notag
    \end{align}
    Applying Lemma~\ref{lem: prod ratio} we have that the ratio of products tends to $1$ in probability, while using Taylor expansion inside of the exponent we derive
    \[
    \psi\left(\gb+\frac{\theta}{\sqrt{m}}  \right)-\psi(\gb)-\frac{\theta}{\sqrt{m}} \psi'(\gb)=\frac{\theta^2}{2m}\psi''(\gb)+\op(1).
    \]
    Thus the limit in probability of the moment-generating function of $W(\pi)$~\eqref{eq:char func} is  equal to $
    e^{\frac{\theta^2}{2}\psi''(\gb)},
    $
    which is that of a normal random variable with mean zero and variance $\psi''(\gb)$.
\end{proof}

\subsection{Proof of the CLT for the Gibbs average}
Before proving the central limit theorem for the Gibbs average, Theorem~\ref{thm:gibbs CLT}, we first present another expansion for $\wh{Z}_n(\gb)$ similar to that in Lemma~\ref{lem:Zhat}.

\begin{lem}\label{lem:Zhatv2}
Under Assumptions~\ref{as:Graph}--\ref{as:optimization} for any $\gb\in [0,\infty)$, for each edge $e\in E$ define 
\begin{equation}\label{eq:ve}
    V_e:=\left(-\go_e-\psi'(\gb)\right)(1+\xi_e(\gb)),
\end{equation}
Then
\begin{equation}\label{eq:Zhatv2}
\sum_{e\in E} V_e \sum_{\gC\subseteq E\setminus\{e\}}\pr(\gC\cup\{e\}\subseteq \pi)\xi_\gC(\gb)-\wh{Z}(\gb)\cdot\sum_{e\in E}p(1-p\xi_e(\gb))V_e = \op(1).
\end{equation}
\end{lem}

Using Lemma~\ref{lem:Zhatv2} the proof of Theorem~\ref{thm:gibbs CLT} reduces to classical probabilistic limit theorems for independent random variables. 

\begin{proof}[Proof of Theorem~\ref{thm:gibbs CLT}]
Recall that
$
\frac{\partial}{\partial \gb} \log Z(\gb,\mvgo)=\gibbs{-W(\pi)}_\gb
$
and notice that the Gibbs average can be rewritten as
\begin{align}\label{eq:gibbsave}
\gibbs{W(\pi)}_\gb+m\psi'(\gb)&=\frac{\sum_{\pi\in\cS} (W(\pi)+m\psi'(\gb))\exp(\gb\cdot W(\pi)+m\psi(\gb))}{\wh{Z}(\gb)}.
\end{align}
Now we consider the numerator of~\eqref{eq:gibbsave}
    \begin{align*}
        &\E_0\bigl(\sum_{e\in E}\left(-\go_e-\psi'(\gb)\right) \1_{e\in\pi}\cdot\prod_{f\in E}\left(1+\xi_e\cdot\1_{f\in\pi}\right)\bigr)\notag\\
        &=\sum_{e\in E}V_e\cdot\sum_{\gC\subseteq E\setminus\{e\}}\pr(\gC\cup e \subseteq \pi)\xi_\gC
        =\wh{Z}(\gb)\cdot\sum_{e\in E}p(1-p\xi_e)V_e\cdot(1+\op(1))
    \end{align*}
where $V_e$ as in \eqref{eq:ve} and where we used Lemma~\ref{lem:Zhatv2} in the last equality. Combining this with \eqref{eq:gibbsave} we get

\begin{align*}
    \gibbs{W(\pi)}_\gb+m\psi'(\gb)
    &=\sum_{e\in E}p(1-p\xi_e)V_e+\op(1)\\
    &=p\sum_{e\in E} V_e -p^2\sum_{e\in E} V_e\,\xi_e+\op(1)\\
    &=\sqrt{\frac{m^2}{E^2}}\sum_{e\in E} V_e-\frac{m^2}{E}\frac{1}{E}\sum_{e\in E} V_e\,\xi_e+\op(1)
    \Rightarrow \N(0,2\gc\E V_e^2)-2\gc\E V_e\,\xi_e,
\end{align*}
where we used CLT for independent random variables alongside the SLLN. Recall that $\E e^{-\gb\go}=e^{\psi(\gb)}$ and thus $\E -\go e^{-2\gb\go}=\psi'(2\gb)e^{\psi(2\gb)}$. Using that we compute the expectations that appear above,
\begin{align*}
\E V_e^2 &=\E((-\go_e-\psi'(\gb))^2\cdot e^{-2\gb\go_e-2\psi(\gb)})=\left[ (\psi'(2\gb)-\psi'(\gb))^2+\psi''(2\gb)\right]\cdot e^{\psi(2\gb)-2\psi(\gb)},    
\end{align*}
and
\begin{align*}
\E V_e\,\xi_e &=\E(-\go_e-\psi'(\gb))\cdot e^{-2\gb\go_e-2\psi(\gb)}
=(\psi'(2\gb)-\psi'(\gb))\cdot e^{\psi(2\gb)-2\psi(\gb)}.
\end{align*}
This completes the proof.
\end{proof}

We end this section with the proof of Lemma~\ref{lem:Zhatv2}.

\begin{proof}[Proof of Lemma~\ref{lem:Zhatv2}]
Recall that $\E V_e=0$ and that $\wh{Z}_n(\gb)=\sum_{\gC\subseteq E}\pr(\gC\subseteq \pi)\xi_\gC.$
Thus, the difference inside of the square in~\eqref{eq:Zhatv2} can be rewritten as
    \begin{align}\label{eq:lem52}
        \sum_{e\in E} V_e &\sum_{\gC\subseteq E\setminus\{e\}}\pr(\gC\cup\{e\}\subseteq \pi)\xi_\gC-\sum_{e\in E}p(1-p\xi_e)V_e\sum_{\gC\subseteq E}\pr(\gC\subseteq \pi)\xi_\gC\\
        =\sum_{e\in E} V_e &\Biggl[\sum_{\gC\subseteq E\atop e\in \gC}\pr(\gC\subseteq \pi)\xi_{\gC\setminus\{e\}}-p\xi_e\sum_{\gC\subseteq E\atop e\in \gC}\pr(\gC\subseteq \pi)\xi_{\gC\setminus\{e\}}\notag\\
        &-p\sum_{\gC\subseteq E\atop e\notin \gC}\pr(\gC\subseteq \pi)\xi_{\gC\setminus\{e\}}+p^2\xi_e^2\sum_{\gC\subseteq E\atop e\notin \gC}\pr(\gC\subseteq \pi)\xi_{\gC\setminus\{e\}}+ p^2\xi_e^2\sum_{\gC\subseteq E\atop e\in \gC}\pr(\gC\subseteq \pi)\xi_{\gC\setminus\{e\}}\Biggr].\notag
    \end{align}
Regrouping the terms yields that
    \begin{align}
        \textnormal{\eqref{eq:lem52}}=\sum_{e\in E}\sum_{\gC\subseteq E\atop e\in \gC}&\Biggl[\left(\pr(\gC\subseteq \pi)-p\pr(\gC\setminus\{e\}\subseteq \pi)\right) V_e\,\xi_{\gC\setminus\{e\}}\label{eq:Zhatv2 Err1}\\
        &-p\left(\pr(\gC\subseteq \pi)-p\pr(\gC\setminus\{e\}\subseteq \pi)\right) V_e\xi_e\, \xi_{\gC\setminus\{e\}}\label{eq:Zhatv2 Err2}\\
        &\qquad+ p^2\xi_e^2 V_e \pr(\gC\subseteq \pi)\,\xi_{\gC\setminus\{e\}}\Biggr].\label{eq:Zhatv2 Err3}
    \end{align}
    We now show that \eqref{eq:Zhatv2 Err1} and \eqref{eq:Zhatv2 Err2} is asymptotically negligible in $L_2$. 
    To simplify the notation further, let 
    \[
    A(\gC,e):=\pr(\gC\subseteq \pi)-p\pr(\gC\setminus\{e\}\subseteq \pi).
    \]
    We first break term~\eqref{eq:Zhatv2 Err1} into two cases depending on whether $\gC$ is a partial matching or not, akin to the way we did it in the proof of Lemma~\ref{lem:Zhat}. In the former case, the fact that $\gC\in \wt{\cM}$ implies that $\gC\setminus\{e\}\in \wt{\cM}$ and thus Assumption~\ref{as:prod of edges} yields that $\abs{A(\gC,e)}^2\le \rho^{\abs{\gC}}\cdot p^{2\abs{\gC}+1}$ for some constant $\rho\in(0,\infty)$.
    
    Although the analysis here is similar to the one presented in the proof of Lemma~\ref{lem:Zhat}, there is one important difference. Namely that $V_e$ and $\xi_{\gC\setminus\{e\}}$ are uncorrelated only for a fixed edge $e$. Therefore, additional care is required for the cross terms.
    \begin{align}
    \norm{\eqref{eq:Zhatv2 Err1}}_2^2 &=\biggl\Vert\sum_{e\in E}\sum_{\gC\in \cM, e\in \gC}A(\gC,e) V_e\,\xi_{\gC\setminus\{e\}}\biggr\Vert_2^2
    =\sum_{e\in E}\sum_{\gC\in \cM\atop e\in \gC}A(\gC,e)^2\E V_e^2\,\xi_{\gC\setminus\{e\}}^2\label{eq:Zhatv2 Err1_1}\\
    &\qquad\qquad\qquad\qquad+\sum_{e\in E, \gC\in \cM\atop e\in \gC}\sum_{e'\in E, \gC'\in \cM\atop e\in \gC}A(\gC,e)A(\gC',e')\E V_eV_{e'}\,\xi_{\gC\setminus\{e\}}\xi_{\gC'\setminus\{e'\}}.\label{eq:Zhatv2 Err1_2}
    \end{align}
    The first part, \eqref{eq:Zhatv2 Err1_1}, is of order $O(1/m)$ by an analogous computation to those in \eqref{eq:dif gen subseteq}. To bound the second part, the cross terms, we notice that the only time the expectation inside of \eqref{eq:Zhatv2 Err1_2} is nonzero is when $\gC=\gC'$ and $e\neq e'$. Even though for each $\gC$ there would be $\abs{\gC}\cdot (\abs{\gC}-1)$ such terms, the exact computations apply again, yielding an upper bound of order $O(m^{-1})$.
    
    The case when $\gC\notin \wt{\cM}$ is treated similarly, with the only difference that now there are fewer choices for picking such a cluster. So, adopting a similar analysis to that in \eqref{eq:dif gen not subseteq}, we retrieve an upper bound of order $O(m^{-1})$. Thus, $\norm{\eqref{eq:Zhatv2 Err1}}_2=o(1).$

    Now, we focus on the term~\eqref{eq:Zhatv2 Err2}. The first difficulty is that $\E V_e\xi_e$ need not be mean zero; hence, we add and subtract its mean.
    \[
    V_e\xi_e\xi_{\gC\setminus\{e\}}= (V_e\xi_e-\mu_\gb)\xi_{\gC\setminus\{e\}} +\mu_\gb\xi_{\gC\setminus\{e\}},
    \]
    where $\mu_\gb:=\E V_e\,\xi_e$.
    This results in~\eqref{eq:Zhatv2 Err2} being separated into two parts
    \begin{align}
    \norm{\eqref{eq:Zhatv2 Err2}}_2^2&\le p^2\biggl\Vert\sum_{e\in E}\sum_{\gC\subseteq E\atop e\in \gC}
    A(\gC,e) (V_e\xi_e-\mu_\gb)\xi_{\gC\setminus\{e\}} \biggr\Vert_2^2\label{eq:Zhatv2 Err2_1}\\
    &\qquad\qquad +(p\mu_\gb)^2\biggl\Vert\sum_{e\in E}\sum_{\gC\subseteq E\atop e\in \gC}  A(\gC,e) \xi_{\gC\setminus\{e\}} \biggr\Vert_2^2\label{eq:Zhatv2 Err2_2}.
    \end{align}
    
    The first part, \eqref{eq:Zhatv2 Err2_1}, is treated in exactly the same way as the term~\eqref{eq:Zhatv2 Err1} as it features a product of centered and, for a fixed edge $e$, uncorrelated random variables.
    The second part, \eqref{eq:Zhatv2 Err2_2}, is slightly more subtle. The analysis is still similar; however, since edge $e$ does not play a role inside the sum, we can rewrite it as a sum over clusters $\wh{\gC}$ that do not contain edge $e$. This results in an extra factor of $(E-|\wh{\gC}|)$. However, since $p^2=\Theta(E^{-1})$, the extra factor of $p^2$ in front of \eqref{eq:Zhatv2 Err2_2} cancels out that contribution. Indeed, in case when the cluster is a partial matching 
    \begin{align*}
        \eqref{eq:Zhatv2 Err2_2}
        &=(p\mu_\gb)^2\norm{\sum_{\wh{\gC}\in \wt{\cM}, e\notin \wh{\gC}}  A\left(\wh{\gC}\cup e,e\right)\, \xi_{\wh{\gC}} }_2^2\\
        &=(p\mu_\gb)^2\sum_{k=0}^{E-1}\sum_{\wh{\gC}\in \wt{\cM}\atop \abs{\wh{\gC}}=k}\sum_{e\notin \wh{\gC}} A\left(\wh{\gC}\cup e,e\right)^2v_\gb^{2k}\\
        &\le (p\mu_\gb)^2\sum_{k=0}^{E-1} {E\choose k}(E-k)^2p^{2(k+1)+1}v_\gb^{2k}\rho^k
        =O(E^{-1}),
    \end{align*}
    where in the second to last inequality we used Assumption~\ref{as:prod of edges}~for cluster $\wh{\gC}\cup e$, which is of the size $k+1$. The case when $\wh{\gC}\notin \wt{\cM}$ is treated in a similar fashion to before, relying on the Assumption~\ref{as:non matching}~and the fact that there are fewer choices to pick such a cluster.
    Thus, we conclude that $\norm{\eqref{eq:Zhatv2 Err2}}_2=o(1).$

    Finally, we show the last term~\eqref{eq:Zhatv2 Err3} is small in $L_1$ via a direct computation
    \begin{align*}
        \norm{\eqref{eq:Zhatv2 Err3}}_1 &\le p^2\sum_{e\in E}\norm{\xi_e^2 V_e}_1 \cdot \biggl\Vert\sum_{\gC\subseteq E, e\in \gC}\pr(\gC\subseteq \pi)\xi_{\gC\setminus\{e\}}\biggr\Vert_2.
    \end{align*}
    By the same argument as before 
    \[
\biggl\Vert\sum_{\gC\subseteq E,e\in \gC}\pr(\gC\subseteq \pi)\xi_{\gC\setminus\{e\}}\biggr\Vert_2=O(E^{-1}).
    \]
    Recall that $p=m/E=O(E^{-1/2})$. Combining that with the bound about, we get that $\norm{\eqref{eq:Zhatv2 Err3}}_1=O(E^{-1/2})$. This completes the proof.
\end{proof}

\begin{proof}[Proof of Theorem~\ref{thm:log part cltp}]
The proof of the result follows similar steps to those in the proof of Theorem~\ref{thm:log part clt}. The only distinction in the analysis is that one works with each block $E_{s,t}$ separately.  Namely, one can derive the expansion analogous to the one in Lemma~\ref{lem:Zhat}
\[
\E \bigl|\prod_{s\le t}\prod_{e\in E_{s,t}}\left(1+p_{s,t}\xi_e(\gb)\right)-\wh{Z}(\gb)\bigr|^2\lesssim \frac{1}{\max_{s\le t}m_{s,t}}\approx \frac{1}{n}.
\]
\end{proof}

\section{Closing remarks and further work}\label{sec future}
In this section, we outline several possible directions for future research, emphasizing interesting by-products of our results and techniques. 

\subsection{Directed polymer model in the intermediate disorder regime}\label{sec:polymer}
In Section~\ref{intro cluster}, we noted that several works (see, for example, \cite{AlbertsKhaninQuastel,CaravennaSunZygouras,DZ16}) studying the directed polymer model in dimensions $(1+1)$ and $(2+1)$ under the intermediate disorder regime encountered an obstacle similar to the one we faced when applying the chaos expansion. Although these authors overcame it by requiring $\beta_n$ to tend to zero sufficiently fast, we believe that our techniques could also be applied in their setting. Higher dimensions $(d+1)$ with $d \ge 3$ have likewise been investigated in \cite{CometsYoshida} using the martingale method. Here too, we expect that our approach could be adapted to yield an alternative, combinatorial proof of their results.

\subsection{The minimal star and the minimal constellation problems}\label{sec:star}
    This problem highlights a potential necessity of our Assumption~\ref{as:matching},~which requires that any partial matching can be extended to a valid configuration, \ie\ an element of $\cS$, in an optimization problem that we consider. Moreover, we crucially relied that matchings, in a sense, have a dominant contribution in any configuration, which weakens the dependence between the edges and enables CLT. A natural next question is: \textit{In the setting of the present paper, for fixed or growing parameter $\beta$, is this assumption necessary for the Gaussian limiting distribution in random optimization problems?}

    We believe that at least an approximate version is needed; in other words, one has to assume that for any partial matching of size up to some $\kappa(n)$, that tends to infinity as $n$ grows, there has to be a configuration in $\cS$ that contains it. In the case of Assumption~\ref{as:matching} we take the largest $\kappa(n)=\abs{V}/2$. 

    In fact, a family of optimization problems might exhibit a phase transition in terms of $\kappa(n)$. For example, consider the following problem, which we call \textit{the Minimal Star Problem}.
    
    Let $G$ be the complete graph on $n$ vertices and $\{\go_e\}_{e\in E}$ be an i.i.d.~sequence of $\Exp(1)$ random variables. Suppose $\cS$ is the collection of stars, \ie~the subgraphs formed by selecting a single vertex together with all of the edges adjacent to it. Notice that $\cS$ does not satisfy Assumption~\ref{as:matching}, because only matching of size $\kappa(n)\equiv 1$ can be extended to a valid configuration. As before, define the weight of a configuration (star) to be the sum of the weights on its edges and denote it by
    \[
        S_i=\sum_{i\in[n]} \go_{ij}
    \]
    and the minimal weight by $S_{\min}:=\min_{i\in [n]}S_i.$

    Notice that the weight of each star is a sum of $(n-1)$ i.i.d.~random variables and any two stars intersect in exactly one edge. Thus, one expects that for suitable parameters $b$ and $m$
    \[
    \sqrt{\frac{2\log n}{n-1}}\cdot (S_{\min}-n+1)+2\log n\Rightarrow \textnormal{Gumbel}(m,b).
    \]

    From the minimal star problem, it is easy to come up with a whole family of optimization problems, such that a valid configuration contains only partial matching of sizes not exceeding $\kappa\in[n]$. For example, consider selecting $\kappa$ vertices and consider the union of vertex-disjoint stars centered at these vertices, each containing $(n-\kappa)/\kappa$-many leaves. We call such graphs \textit{constellations}, thus leading to the name \textit{the minimal ($\kappa$-)constellation problem}. Notice that when $\kappa=1$, this problem reduces to the minimal star problem described above, while when $\kappa=n/2$, it reduces to MMP. Since in the former case, one expects Gumbel limiting distribution, while in the latter case, Gaussian, it is natural to ask the following problem.

\begin{problem}\label{prob:star}
    Find $\kappa_c$ such that for any $\kappa(n)<\kappa_c$ the limiting distribution of appropriately scaled and centered weight of the minimal  $\kappa$-constellation is Gumbel, and $\kappa_c>\kappa(n)$ it is Gaussian. 
\end{problem}

This problem makes sense for both versions: true optimization as well as Gibbs version with fixed or growing inverse temperature parameter $\gb$. In the case of true optimization, Gaussian limiting behavior for MMP is still open~\cite[Conjecture 1.1]{wastlund2005variance}.

\subsection{First passage percolation on the complete graph}

    Throughout this paper we assumed that one optimizes over a family of spanning graphs with the same size of the edge set equal to $m=\Theta(\abs{V})$, see  Assumption~\ref{as:kappa}.
    In fact, we used it to establish that $\E Z(\gb)=e^{m\cdot \psi(\gb)}.$ However, many optimization problems of interest do not have this quality. For example, one might be interested in the path of the smallest weight connecting two distinct vertices in a graph, (see~\cite{RemcoFPP,JansonFPP,ShankarRemcoFPPHop} among others). In the literature, this problem is referred to as the first passage time or chemical distance between two points. Let $G$ be the complete graph on $n$ vertices and $\{\go_e\}_{e\in E}$ be an i.i.d.~sequence of $\Exp(1)$ random variables. Suppose $\cS$ is the collection of all paths connecting two distinguished vertices. Clearly, they are not all of the same size. Moreover, such configurations are not spanning. Thus, this problem is far from the scope of the current paper. In~\cite{ShankarRemcoFPPHop}, Bhamidi and van der Hofstad established that the number of edges in the optimal path, known as the \textit{hop-count}, has mean and variance of order $\log(n)$ and obeys CLT. They also established that properly resealed and centered total weight of the optimal path is asymptotically Gumbel. It is natural to ask how limiting behavior changes if one imposes additional conditions on the allowed lengths (hop counts) of the paths. For instance, if paths have to be of length $(n-1)$ then the problem reduces to TSP and, in the case of fixed temperatures, is covered by our result. This suggests that there might be a phase transition similar to the one mentioned in Problem~\ref{prob:star}.

\subsection{From fixed temperature to true optimization}
In this article, we considered an auxiliary Gibbs measure, with a fixed inverse temperature parameter, on the set of all subgraphs over which one is optimizing. Thus, the resulting process still retains a complicated dependence structure of the original optimization problem but is merely biased towards the typical weight. Naturally, it is of interest to consider the case when $\gb$ tends to infinity. In this case, the cluster expansion would involve dependent random variables, which would not be convergent. 

Consider the case when $\go\sim \textrm{Exp}(1)$, we have
$\E e^{-\gb\go}=(1+\gb)^{-1},\psi(\gb)=-\log(1+\gb)$
and
$v_\gb^2=\gb^2/(1+2\gb)\approx \gb/2$
as $\gb\to\infty$. In general, if $\go$ is positive almost surely with $\pr(\go\le x)\approx x^\gk$ as $x\downarrow 0$ for some $\gk>0$, we have, 
\[
\E e^{-\gb\go}= \int_0^\infty e^{-x}\pr(\go\le x/\gb)dx \approx  \gC(1+\gk)\cdot \gb^{-\gk}
\]
for large $\gb$. In particular, $v_\gb^2 \approx \gb^{\gk}/(2^\gk\gC(1+\gk))$ for large $\gb$. Plugging this into Theorem~\ref{thm:log part clt} we get 
\[
\gb^{-\gk/2}(\log Z(\gb)+\gk m\log \gb +\gc \gb^{\gk}/2^\gk\gC(1+\gk)) \approx \N(0,2\gc/(2^\gk\gC(1+\gk))).
\]
In the scope of our work, $\gb$ is a constant. We expect our techniques to work similarly when $\gb$ tends to infinity sufficiently slow. However, if we allow $\gb=(\wh{\gb}\cdot n)^{1/\gk}$ (as expected from~\cite{WastlundReplicaRA}), and recall that $\log Z(\gb)/\gb$ convergences to $-M_n$, where $M_n$ is the minimum energy, one can expect the random variable
$
n^{1/\gk-1/2}\cdot (M_n - \E M_n)
$
to converge to a Gaussian distribution. We believe that investigating distributional limit in terms of general $\kappa$ and growing $\beta$ is challenging and requires novel ideas. 

\subsection{General block models}\label{sec:genblock}

It is of interest to generalize the approach used in this article to cases where $\cS$ consists of subgraphs of different sizes. For example, one can consider an $\ell$--partite graph with $\ell\ge2$ many blocks of sizes $n_1,n_2,\ldots,n_\ell$, however $\vm$ is not fixed. Although we expect such graphs to fall into our framework additional challenges arise. Firstly, the combinatorics prove to be more complicated as most known literature to us is conspired with studying the existence of the perfect matching or a Hamiltonian path in these more general graphs and not the numbers of these objects in them. Secondly, the $m$ edges of a valid configuration could be divided in a number of ways between the blocks creating an additional layer to the dependency and thus we can no longer take advantage of some of the convenient cancellations. However, one can combine our results for fixed $\vm$ from Section~\ref{sec:block models} with large deviation results for $\vm$ to handle such situations, which we leave for future research.


\section*{Acknowledgments}  
We thank Robert Krueger for suggesting the proof of Lemma~\ref{th:cayley}, and Qiang Wu for many enlightening discussions at the beginning stage of this project. Research of the first author was partially supported by Campus Research Board Grant RB23016. The second author was supported in part by the RTG award grant (DMS-2134107) from the NSF.

\bibliographystyle{imsart-nameyear}
\bibliography{opt.bib}
\end{document}

%% file: style.tex
\usepackage{enumitem}
\setlist[enumerate,1]{label={\textsc{\arabic*.}}}

\usepackage{mathrsfs,xfrac} 
\usepackage[colorlinks=true,linkcolor=blue,citecolor=blue,urlcolor=blue]{hyperref} 
\usepackage{amsmath,amssymb,amsthm,amsfonts,amsbsy,latexsym,dsfont,color,graphicx,enumitem}
\usepackage[numeric,initials,nobysame]{amsrefs} 
\usepackage{caption}
\usepackage[textsize=tiny]{todonotes}
\usepackage{regexpatch}
\usepackage{float}
\makeatletter
\xpatchcmd{\@todo}
\makeatother
\usepackage{comment}

\newtheorem{thm}{Theorem}[section]
\newtheorem{lem}[thm]{Lemma}
\newtheorem{cor}[thm]{Corollary}

\newtheorem{obs}[thm]{Observation}

\newtheorem{ass}{Assumption}

\newtheorem{innercustomthm}{Assumption}
\newenvironment{customthm}[1]
  {\renewcommand\theinnercustomthm{#1}\innercustomthm}
  {\endinnercustomthm}

\newtheorem{problem}{Problem}

\theoremstyle{definition}
\newtheorem{rem}[thm]{Remark}

\renewcommand{\le}{\leqslant}  
\renewcommand{\ge}{\geqslant}

\newcommand{\wt}{\widetilde}

\newcommand{\ind}{\mathds{1}}
\newcommand{\1}{\mathds{1}}
\newcommand{\eps}{\varepsilon}

\newcommand{\norm}[1]{\left\Vert#1\right\Vert}
\newcommand{\abs}[1]{\left\vert#1\right\vert}
\newcommand{\gibbs}[1]{\left\langle#1\right\rangle}

\newcommand{\ie}{\emph{i.e.,}}

\newcommand{\eg}{\emph{e.g.,}}

\let\ga=\alpha \let\gb=\beta \let\gc=\gamma  
    \let\gk=\kappa \let\gl=\lambda      \let\go=\omega   \let\gs=\sigma  
  
\let\gC=\Gamma

\newcommand{\cM}{\mathcal{M}}
\newcommand{\cP}{\mathcal{P}}
\newcommand{\cS}{\mathcal{S}}\newcommand{\cT}{\mathcal{T}}

\newcommand{\vm}{\mathbf{m}}

\newcommand{\mvs}{\boldsymbol{s}}

\newcommand{\mvgo}{\boldsymbol{\omega}}

\newcommand{\bN}{\mathbb{N}}
\newcommand{\bR}{\mathbb{R}}

\newcommand{\dR}{\mathds{R}}

\DeclareMathOperator{\E}{\mathds{E}}
\DeclareMathOperator{\pr}{\mathds{P}}

\DeclareMathOperator{\var}{Var}

\DeclareMathOperator{\Exp}{Exp}

\DeclareMathOperator{\N}{N}
\newcommand{\wh}[1]{\widehat{#1}}

\providecommand{\dtv}{\textrm{d}_{\mathrm{TV}}}

\providecommand{\poi}{\mathrm{Poisson}}
\providecommand{\op}{\mathrm{o}_{\mathrm{p}}}


%% file: opt.bib
@article {CaravennaSunZygouras,
    AUTHOR = {Caravenna, Francesco and Sun, Rongfeng and Zygouras, Nikos},
     TITLE = {The two-dimensional {KPZ} equation in the entire subcritical
              regime},
   JOURNAL = {Ann. Probab.},
  FJOURNAL = {The Annals of Probability},
    VOLUME = {48},
      YEAR = {2020},
    NUMBER = {3},
     PAGES = {1086--1127},
      ISSN = {0091-1798,2168-894X},
   MRCLASS = {60H15 (35K59 35R60 82B44 82D60)},
  MRNUMBER = {4112709},
MRREVIEWER = {Jiang\ Lun\ Wu},
       DOI = {10.1214/19-AOP1383},
       URL = {https://doi-org.libproxy.lib.unc.edu/10.1214/19-AOP1383},
}

@article {AlbertsKhaninQuastel,
    AUTHOR = {Alberts, Tom and Khanin, Konstantin and Quastel, Jeremy},
     TITLE = {The intermediate disorder regime for directed polymers in
              dimension {$1+1$}},
   JOURNAL = {Ann. Probab.},
  FJOURNAL = {The Annals of Probability},
    VOLUME = {42},
      YEAR = {2014},
    NUMBER = {3},
     PAGES = {1212--1256},
      ISSN = {0091-1798,2168-894X},
   MRCLASS = {60F05 (60K37 82D10)},
  MRNUMBER = {3189070},
MRREVIEWER = {Patr\'icia\ Gon\c calves},
       DOI = {10.1214/13-AOP858},
       URL = {https://doi-org.libproxy.lib.unc.edu/10.1214/13-AOP858},
}

@article {JansonCLTnumber,
    AUTHOR = {Janson, Svante},
     TITLE = {The numbers of spanning trees, {H}amilton cycles and perfect
              matchings in a random graph},
   JOURNAL = {Combin. Probab. Comput.},
  FJOURNAL = {Combinatorics, Probability and Computing},
    VOLUME = {3},
      YEAR = {1994},
    NUMBER = {1},
     PAGES = {97--126},
      ISSN = {0963-5483,1469-2163},
   MRCLASS = {05C80 (60C05)},
  MRNUMBER = {1285693},
MRREVIEWER = {Andrzej\ Ruci\'nski},
       DOI = {10.1017/S0963548300001012},
       URL = {https://doi.org/10.1017/S0963548300001012},
}

@article {ScolaClust,
    AUTHOR = {Scola, Giuseppe},
     TITLE = {Cluster expansion for the {I}sing model in the canonical
              ensemble},
   JOURNAL = {Math. Phys. Anal. Geom.},
  FJOURNAL = {Mathematical Physics, Analysis and Geometry. An International
              Journal Devoted to the Theory and Applications of Analysis and
              Geometry to Physics},
    VOLUME = {24},
      YEAR = {2021},
    NUMBER = {2},
     PAGES = {Paper No. 8, 33},
      ISSN = {1385-0172,1572-9656},
   MRCLASS = {82B20 (82B05)},
  MRNUMBER = {4235355},
       DOI = {10.1007/s11040-021-09377-3},
       URL = {https://doi.org/10.1007/s11040-021-09377-3},
}

@article {JenssenClustTorus,
    AUTHOR = {Jenssen, Matthew and Keevash, Peter},
     TITLE = {Homomorphisms from the torus},
   JOURNAL = {Adv. Math.},
  FJOURNAL = {Advances in Mathematics},
    VOLUME = {430},
      YEAR = {2023},
     PAGES = {Paper No. 109212, 89},
      ISSN = {0001-8708,1090-2082},
   MRCLASS = {05C30 (05C31 82B20)},
  MRNUMBER = {4619447},
MRREVIEWER = {Ioan\ Tomescu},
       DOI = {10.1016/j.aim.2023.109212},
       URL = {https://doi.org/10.1016/j.aim.2023.109212},
}

@article {JenssenClust,
    AUTHOR = {Jenssen, Matthew and Perkins, Will},
     TITLE = {Independent sets in the hypercube revisited},
   JOURNAL = {J. Lond. Math. Soc. (2)},
  FJOURNAL = {Journal of the London Mathematical Society. Second Series},
    VOLUME = {102},
      YEAR = {2020},
    NUMBER = {2},
     PAGES = {645--669},
      ISSN = {0024-6107,1469-7750},
   MRCLASS = {05C69 (05C30 05C31 82B20)},
  MRNUMBER = {4171429},
MRREVIEWER = {J\'{o}zsef\ Balogh},
       DOI = {10.1112/jlms.12331},
       URL = {https://doi.org/10.1112/jlms.12331},
}

@article {HelmuthClust,
    AUTHOR = {Helmuth, Tyler and Jenssen, Matthew and Perkins, Will},
     TITLE = {Finite-size scaling, phase coexistence, and algorithms for the
              random cluster model on random graphs},
   JOURNAL = {Ann. Inst. Henri Poincar\'{e} Probab. Stat.},
  FJOURNAL = {Annales de l'Institut Henri Poincar\'{e} Probabilit\'{e}s et
              Statistiques},
    VOLUME = {59},
      YEAR = {2023},
    NUMBER = {2},
     PAGES = {817--848},
      ISSN = {0246-0203,1778-7017},
   MRCLASS = {82B20 (05C80 60J10 82B26)},
  MRNUMBER = {4575018},
MRREVIEWER = {Simone\ Baldassarri},
       DOI = {10.1214/22-aihp1263},
       URL = {https://doi.org/10.1214/22-aihp1263},
}

@article {KhaninClust,
    AUTHOR = {Khanin, K. M. and Lebovits, Dzh. L. and Mazel , A. E.
              and Sina\u{\i}, Ya. G.},
     TITLE = {Self-avoiding random walks in five or more dimensions: an
              approach using polymer expansions},
   JOURNAL = {Uspekhi Mat. Nauk},
  FJOURNAL = {Uspekhi Matematicheskikh Nauk},
    VOLUME = {50},
      YEAR = {1995},
    NUMBER = {2(302)},
     PAGES = {175--206},
      ISSN = {0042-1316,2305-2872},
   MRCLASS = {60J15 (82B41)},
  MRNUMBER = {1339268},
MRREVIEWER = {Andr\'{a}s\ Kr\'{a}mli},
       DOI = {10.1070/RM1995v050n02ABEH002085},
       URL = {https://doi.org/10.1070/RM1995v050n02ABEH002085},
}

@article {ParthaQiang,
    AUTHOR = {Dey, Partha S. and Wu, Qiang},
     TITLE = {Mean field spin glass models under weak external field},
   JOURNAL = {Comm. Math. Phys.},
  FJOURNAL = {Communications in Mathematical Physics},
    VOLUME = {402},
      YEAR = {2023},
    NUMBER = {2},
     PAGES = {1205--1258},
      ISSN = {0010-3616,1432-0916},
   MRCLASS = {82B26 (60K35 82D30)},
  MRNUMBER = {4627320},
       DOI = {10.1007/s00220-023-04742-5},
       URL = {https://doi.org/10.1007/s00220-023-04742-5},
}

@article {BanerjeeClust,
    AUTHOR = {Banerjee, Debapratim},
     TITLE = {Fluctuation of the free energy of {S}herrington-{K}irkpatrick
              model with {C}urie-{W}eiss interaction: the paramagnetic
              regime},
   JOURNAL = {J. Stat. Phys.},
  FJOURNAL = {Journal of Statistical Physics},
    VOLUME = {178},
      YEAR = {2020},
    NUMBER = {1},
     PAGES = {211--246},
      ISSN = {0022-4715,1572-9613},
   MRCLASS = {82B44 (60K35)},
  MRNUMBER = {4056655},
MRREVIEWER = {Giuseppe\ Genovese},
       DOI = {10.1007/s10955-019-02428-8},
       URL = {https://doi.org/10.1007/s10955-019-02428-8},
}

@article {AizenmanClust,
    AUTHOR = {Aizenman, M. and Lebowitz, J. L. and Ruelle, D.},
     TITLE = {Some rigorous results on the {S}herrington-{K}irkpatrick spin
              glass model},
   JOURNAL = {Comm. Math. Phys.},
  FJOURNAL = {Communications in Mathematical Physics},
    VOLUME = {112},
      YEAR = {1987},
    NUMBER = {1},
     PAGES = {3--20},
      ISSN = {0010-3616,1432-0916},
   MRCLASS = {82A57 (82A05)},
  MRNUMBER = {904135},
MRREVIEWER = {Aernout\ C. D. van Enter},
       URL = {http://projecteuclid.org/euclid.cmp/1104159806},
}

@article {FarisCluster,
    AUTHOR = {Faris, William G.},
     TITLE = {Combinatorics and cluster expansions},
   JOURNAL = {Probab. Surv.},
  FJOURNAL = {Probability Surveys},
    VOLUME = {7},
      YEAR = {2010},
     PAGES = {157--206},
      ISSN = {1549-5787},
   MRCLASS = {82B20 (05A15 60K35 82-02 82B05)},
  MRNUMBER = {2684165},
MRREVIEWER = {Marco\ Corgini},
       DOI = {10.1214/10-PS159},
       URL = {https://doi.org/10.1214/10-PS159},
}

@article {Sicuro_kfact,
    AUTHOR = {Sicuro, Gabriele and Zdeborov\'{a}, Lenka},
     TITLE = {The planted {$k$}-factor problem},
   JOURNAL = {J. Phys. A},
  FJOURNAL = {Journal of Physics. A. Mathematical and Theoretical},
    VOLUME = {54},
      YEAR = {2021},
    NUMBER = {17},
      ISSN = {1751-8113,1751-8121},
       DOI = {10.1088/1751-8121/abee9d},
       URL = {https://doi.org/10.1088/1751-8121/abee9d},
}

@article {RheeTalagrandTSP,
    AUTHOR = {Rhee, WanSoo T. and Talagrand, Michel},
     TITLE = {A sharp deviation inequality for the stochastic traveling
              salesman problem},
   JOURNAL = {Ann. Probab.},
  FJOURNAL = {The Annals of Probability},
    VOLUME = {17},
      YEAR = {1989},
    NUMBER = {1},
     PAGES = {1--8},
      ISSN = {0091-1798,2168-894X},
   MRCLASS = {60F10 (60G42 68Q25 90B15 90C35)},
  MRNUMBER = {972767},
MRREVIEWER = {Klaus\ D.\ Schmidt},
       URL =
              {http://links.jstor.org/sici?sici=0091-1798(198901)17:1<1:ASDIFT>2.0.CO;2-1&origin=MSN},
}

@article {WastlundTSP,
    AUTHOR = {W\"{a}stlund, Johan},
     TITLE = {The mean field traveling salesman and related problems},
   JOURNAL = {Acta Math.},
  FJOURNAL = {Acta Mathematica},
    VOLUME = {204},
      YEAR = {2010},
    NUMBER = {1},
     PAGES = {91--150},
      ISSN = {0001-5962,1871-2509},
       URL = {https://doi.org/10.1007/s11511-010-0046-7},
}

@article{wastlund2006traveling,
  title={The travelling salesman problem in the stochastic mean field model},
  author={W{\"a}stlund, Johan},
  journal={preprint},
  year={2006}
}

@article{wastlund2006limit,
  title={The limit in the mean field bipartite travelling salesman problem},
  author={W{\"a}stlund, Johan},
  journal={preprint},
  year={2006}
}

@article {FriezeTSP,
    AUTHOR = {Frieze, Alan},
     TITLE = {On random symmetric travelling salesman problems},
   JOURNAL = {Math. Oper. Res.},
  FJOURNAL = {Mathematics of Operations Research},
    VOLUME = {29},
      YEAR = {2004},
    NUMBER = {4},
     PAGES = {878--890},
      ISSN = {0364-765X,1526-5471},
   MRCLASS = {05C80 (60C05 68W40 90C27 90C59)},
  MRNUMBER = {2104159},
MRREVIEWER = {Konrad\ Engel},
       DOI = {10.1287/moor.1040.0105},
       URL = {https://doi.org/10.1287/moor.1040.0105},
}

@article{karp1985probabilistic,
  title={Probabilistic analysis of heuristics},
  author={Karp, Richard M and Steele, J Michael},
  journal={The traveling salesman problem},
  pages={181--205},
  year={1985},
  publisher={New York: Wiley}
}

@article {FriezeSorkinTSP,
    AUTHOR = {Frieze, Alan and Sorkin, Gregory B.},
     TITLE = {The probabilistic relationship between the assignment and
              asymmetric traveling salesman problems},
   JOURNAL = {SIAM J. Comput.},
  FJOURNAL = {SIAM Journal on Computing},
    VOLUME = {36},
      YEAR = {2006/07},
    NUMBER = {5},
     PAGES = {1435--1452},
      ISSN = {0097-5397,1095-7111},
   MRCLASS = {68Q25 (05C85 60C05 68W40)},
  MRNUMBER = {2284089},
MRREVIEWER = {Hans-Ulrich\ Simon},
       DOI = {10.1137/S0097539701391518},
       URL = {https://doi.org/10.1137/S0097539701391518},
}

@article {DyerFriezeTSP,
    AUTHOR = {Dyer, M. E. and Frieze, A. M.},
     TITLE = {On patching algorithms for random asymmetric travelling
              salesman problems},
   JOURNAL = {Math. Programming},
  FJOURNAL = {Mathematical Programming},
    VOLUME = {46},
      YEAR = {1990},
    NUMBER = {3},
     PAGES = {361--378},
      ISSN = {0025-5610,1436-4646},
   MRCLASS = {90B06 (90C35)},
  MRNUMBER = {1054144},
MRREVIEWER = {Guntram\ Scheithauer},
       DOI = {10.1007/BF01585751},
       URL = {https://doi.org/10.1007/BF01585751},
}

@article {KarpTSP,
    AUTHOR = {Karp, Richard M.},
     TITLE = {A patching algorithm for the nonsymmetric traveling-salesman
              problem},
   JOURNAL = {SIAM J. Comput.},
  FJOURNAL = {SIAM Journal on Computing},
    VOLUME = {8},
      YEAR = {1979},
    NUMBER = {4},
     PAGES = {561--573},
      ISSN = {0097-5397},
   MRCLASS = {90C08 (68C25)},
  MRNUMBER = {573847},
MRREVIEWER = {B.\ Korte},
       DOI = {10.1137/0208045},
       URL = {https://doi.org/10.1137/0208045},
}

@article{vannimenus1984statistical,
  title={On the statistical mechanics of optimization problems of the travelling salesman type},
  author={Vannimenus, Jean and M{\'e}zard, Marc},
  journal={Journal de Physique Lettres},
  volume={45},
  number={24},
  pages={1145--1153},
  year={1984},
  publisher={Les Editions de Physique}
}

@article{mezard1986replicaTSP,
  title={A replica analysis of the travelling salesman problem},
  author={M{\'e}zard, Marc and Parisi, Giorgio},
  journal={Journal de physique},
  volume={47},
  number={8},
  pages={1285--1296},
  year={1986},
  publisher={Soci{\'e}t{\'e} fran{\c{c}}aise de physique}
}

@article{schrijver2005history,
  title={On the history of combinatorial optimization (till 1960)},
  author={Schrijver, Alexander},
  journal={Handbooks in operations research and management science},
  volume={12},
  pages={1--68},
  year={2005},
  publisher={Elsevier}
}

@article{menger1928theorem,
  title={Ein theorem {\"u}ber die bogenl{\"a}nge},
  author={Menger, K},
  journal={Anzeiger—Akademie der Wissenschaften in Wien—Mathematisch-naturwissenschaftliche Klasse},
  volume={65},
  pages={264--266},
  year={1928}
}

@article{monge1781memoire,
  title={M{\'e}moire sur la th{\'e}orie des d{\'e}blais et des remblais},
  author={Monge, Gaspard},
  journal={Mem. Math. Phys. Acad. Royale Sci.},
  pages={666--704},
  year={1781}
}

@book{biggs1986graph,
  title={Graph Theory, 1736-1936},
  author={Biggs, Norman and Lloyd, E Keith and Wilson, Robin J},
  year={1986},
  publisher={Oxford University Press}
}

@article {CMselfloops,
    AUTHOR = {Angel, Omer and van der Hofstad, Remco and Holmgren, Cecilia},
     TITLE = {Limit laws for self-loops and multiple edges in the
              configuration model},
      NOTE = {[Author name corrected by publisher]},
   JOURNAL = {Ann. Inst. Henri Poincar\'{e} Probab. Stat.},
  FJOURNAL = {Annales de l'Institut Henri Poincar\'{e} Probabilit\'{e}s et
              Statistiques},
    VOLUME = {55},
      YEAR = {2019},
    NUMBER = {3},
     PAGES = {1509--1530},
      ISSN = {0246-0203,1778-7017},
   MRCLASS = {60K35 (82B41)},
  MRNUMBER = {4010943},
MRREVIEWER = {Bruno\ Schapira},
       DOI = {10.1214/18-aihp926},
       URL = {https://doi.org/10.1214/18-aihp926},
}

@article {ShankarRemcoFPPHop,
    AUTHOR = {Bhamidi, Shankar and van der Hofstad, Remco},
     TITLE = {Weak disorder asymptotics in the stochastic mean-field model
              of distance},
   JOURNAL = {Ann. Appl. Probab.},
  FJOURNAL = {The Annals of Applied Probability},
    VOLUME = {22},
      YEAR = {2012},
    NUMBER = {1},
     PAGES = {29--69},
      ISSN = {1050-5164,2168-8737},
   MRCLASS = {60C05 (05C80 05C82 60J80 60K35 90B15)},
  MRNUMBER = {2932542},
MRREVIEWER = {David\ J.\ Aldous},
       DOI = {10.1214/10-AAP753},
       URL = {https://doi.org/10.1214/10-AAP753},
}

@incollection {JansonFPP,
    AUTHOR = {Janson, Svante},
     TITLE = {One, two and three times {$\log n/n$} for paths in a complete
              graph with random weights},
      NOTE = {Random graphs and combinatorial structures (Oberwolfach,
              1997)},
   JOURNAL = {Combin. Probab. Comput.},
  FJOURNAL = {Combinatorics, Probability and Computing},
    VOLUME = {8},
      YEAR = {1999},
    NUMBER = {4},
     PAGES = {347--361},
      ISSN = {0963-5483,1469-2163},
   MRCLASS = {05C80 (05C38 60C05)},
  MRNUMBER = {1723648},
MRREVIEWER = {Walter\ M.\ B\"{o}hm},
       DOI = {10.1017/S0963548399003892},
       URL = {https://doi.org/10.1017/S0963548399003892},
}

@article {RemcoFPP,
    AUTHOR = {Eckhoff, Maren and Goodman, Jesse and van der Hofstad, Remco
              and Nardi, Francesca R.},
     TITLE = {Short paths for first passage percolation on the complete
              graph},
   JOURNAL = {J. Stat. Phys.},
  FJOURNAL = {Journal of Statistical Physics},
    VOLUME = {151},
      YEAR = {2013},
    NUMBER = {6},
     PAGES = {1056--1088},
      ISSN = {0022-4715,1572-9613},
   MRCLASS = {60K35 (05C30 60C05)},
  MRNUMBER = {3063496},
MRREVIEWER = {Jonathan\ Henry\ Jordan},
       DOI = {10.1007/s10955-013-0743-7},
       URL = {https://doi.org/10.1007/s10955-013-0743-7},
}

@article {ChatterjeeSenMST,
    AUTHOR = {Chatterjee, Sourav and Sen, Sanchayan},
     TITLE = {Minimal spanning trees and {S}tein's method},
   JOURNAL = {Ann. Appl. Probab.},
  FJOURNAL = {The Annals of Applied Probability},
    VOLUME = {27},
      YEAR = {2017},
    NUMBER = {3},
     PAGES = {1588--1645},
      ISSN = {1050-5164,2168-8737},
   MRCLASS = {60D05 (60B10 60F05)},
  MRNUMBER = {3678480},
MRREVIEWER = {Christoph\ Th\"{a}le},
       DOI = {10.1214/16-AAP1239},
       URL = {https://doi.org/10.1214/16-AAP1239},
}

@book{wastlund2005evaluation,
  title={Evaluation of Janson’s constant for the variance in the random minimum spanning tree problem},
  author={W{\"a}stlund, Johan},
  year={2005},
  publisher={Link{\"o}ping University Electronic Press}
}

@article {JansonMSTCLTadd,
    AUTHOR = {Janson, Svante and W\"{a}stlund, Johan},
     TITLE = {Addendum to: ``{T}he minimal spanning tree in a complete graph
              and a functional limit theorem for trees in a random graph''
              [{R}andom {S}tructures {A}lgorithms {\bf 7} (1995), no. 4,
              337--355; MR1369071] by {J}anson},
   JOURNAL = {Random Structures Algorithms},
  FJOURNAL = {Random Structures \& Algorithms},
    VOLUME = {28},
      YEAR = {2006},
    NUMBER = {4},
     PAGES = {511--512},
      ISSN = {1042-9832,1098-2418},
   MRCLASS = {05C80 (05C35)},
  MRNUMBER = {2225705},
MRREVIEWER = {Andrzej\ Ruci\'{n}ski},
       DOI = {10.1002/rsa.20122},
}

@article {JansonMSTCLT,
    AUTHOR = {Janson, Svante},
     TITLE = {The minimal spanning tree in a complete graph and a functional
              limit theorem for trees in a random graph},
   JOURNAL = {Random Structures Algorithms},
  FJOURNAL = {Random Structures \& Algorithms},
    VOLUME = {7},
      YEAR = {1995},
    NUMBER = {4},
     PAGES = {337--355},
      ISSN = {1042-9832,1098-2418},
   MRCLASS = {05C80 (05C35)},
  MRNUMBER = {1369071},
}

@book{ramey1982non,
  title={A non-parametric test of bimodality with applications to cluster analysis},
  author={Ramey, Daniel Bruce},
  year={1982},
  publisher={Yale University}
}

@incollection {LeeIndep,
    AUTHOR = {Lee, Sungchul and Su, Zhonggen},
     TITLE = {The central limit theorem for the independence number for
              minimal spanning trees in the unit square},
 BOOKTITLE = {Stein's method and applications},
    SERIES = {Lect. Notes Ser. Inst. Math. Sci. Natl. Univ. Singap.},
    VOLUME = {5},
     PAGES = {103--117},
 PUBLISHER = {Singapore Univ. Press, Singapore},
      YEAR = {2005},
      ISBN = {981-256-281-8},
   MRCLASS = {60F05 (60C05 60G55)},
  MRNUMBER = {2205330},
}

@article {LeeMST2,
    AUTHOR = {Lee, Sungchul},
     TITLE = {The central limit theorem for {E}uclidean minimal spanning
              trees. {II}},
   JOURNAL = {Adv. in Appl. Probab.},
  FJOURNAL = {Advances in Applied Probability},
    VOLUME = {31},
      YEAR = {1999},
    NUMBER = {4},
     PAGES = {969--984},
      ISSN = {0001-8678,1475-6064},
   MRCLASS = {60D05 (05C05 60F05 60K35 90C35)},
  MRNUMBER = {1747451},
}

@article {LeeMST1,
    AUTHOR = {Lee, Sungchul},
     TITLE = {The central limit theorem for {E}uclidean minimal spanning
              trees. {I}},
   JOURNAL = {Ann. Appl. Probab.},
  FJOURNAL = {The Annals of Applied Probability},
    VOLUME = {7},
      YEAR = {1997},
    NUMBER = {4},
     PAGES = {996--1020},
      ISSN = {1050-5164,2168-8737},
   MRCLASS = {60D05 (05C05 60F05 60K35)},
  MRNUMBER = {1484795}
}

@article {KestenLee96,
    AUTHOR = {Kesten, Harry and Lee, Sungchul},
     TITLE = {The central limit theorem for weighted minimal spanning trees
              on random points},
   JOURNAL = {Ann. Appl. Probab.},
  FJOURNAL = {The Annals of Applied Probability},
    VOLUME = {6},
      YEAR = {1996},
    NUMBER = {2},
     PAGES = {495--527},
      ISSN = {1050-5164,2168-8737},
   MRCLASS = {60D05 (60F05)},
  MRNUMBER = {1398055},
}

@article {AlexanderMST,
    AUTHOR = {Alexander, Kenneth S.},
     TITLE = {The {RSW} theorem for continuum percolation and the {CLT} for
              {E}uclidean minimal spanning trees},
   JOURNAL = {Ann. Appl. Probab.},
  FJOURNAL = {The Annals of Applied Probability},
    VOLUME = {6},
      YEAR = {1996},
    NUMBER = {2},
     PAGES = {466--494},
      ISSN = {1050-5164,2168-8737},
   MRCLASS = {60K35 (90C27)},
  MRNUMBER = {1398054},
}

@article {SteeleSubadditive,
    AUTHOR = {Steele, J. Michael},
     TITLE = {Subadditive {E}uclidean functionals and nonlinear growth in
              geometric probability},
   JOURNAL = {Ann. Probab.},
  FJOURNAL = {The Annals of Probability},
    VOLUME = {9},
      YEAR = {1981},
    NUMBER = {3},
     PAGES = {365--376},
      ISSN = {0091-1798,2168-894X},
   MRCLASS = {60F15 (60D05 68C05)},
  MRNUMBER = {626571},
}

@article {Avram92,
    AUTHOR = {Avram, Florin and Bertsimas, Dimitris},
     TITLE = {The minimum spanning tree constant in geometrical probability
              and under the independent model: a unified approach},
   JOURNAL = {Ann. Appl. Probab.},
  FJOURNAL = {The Annals of Applied Probability},
    VOLUME = {2},
      YEAR = {1992},
    NUMBER = {1},
     PAGES = {113--130},
      ISSN = {1050-5164,2168-8737},
   MRCLASS = {60D05 (90C27)},
  MRNUMBER = {1143395}
}

@article {RossSurvey,
    AUTHOR = {Ross, Nathan},
     TITLE = {Fundamentals of {S}tein's method},
   JOURNAL = {Probab. Surv.},
  FJOURNAL = {Probability Surveys},
    VOLUME = {8},
      YEAR = {2011},
     PAGES = {210--293},
   MRCLASS = {60F05 (05C80 60C05)},
  MRNUMBER = {2861132},
MRREVIEWER = {Anant P. Godbole},
       DOI = {10.1214/11-PS182},
       URL = {https://doi-org.proxy2.library.illinois.edu/10.1214/11-PS182},
}

@book{talagrand2003spin,
  title={Spin glasses: a challenge for mathematicians: cavity and mean field models},
  author={Talagrand, Michel},
  volume={46},
  year={2003},
  publisher={Springer Science \& Business Media}
}

@inproceedings {SouravSurvey,
    AUTHOR = {Chatterjee, Sourav},
     TITLE = {A short survey of {S}tein's method},
 BOOKTITLE = {Proceedings of the {I}nternational {C}ongress of
              {M}athematicians---{S}eoul 2014. {V}ol. {IV}},
     PAGES = {1--24},
 PUBLISHER = {Kyung Moon Sa, Seoul},
      YEAR = {2014},
      ISBN = {978-89-6105-807-0; 978-89-6105-803-2},
   MRCLASS = {60F05 (60B10)},
  MRNUMBER = {3727600},
}

@article {FriezeMST,
    AUTHOR = {Frieze, A. M.},
     TITLE = {On the value of a random minimum spanning tree problem},
   JOURNAL = {Discrete Appl. Math.},
  FJOURNAL = {Discrete Applied Mathematics. The Journal of Combinatorial
              Algorithms, Informatics and Computational Sciences},
    VOLUME = {10},
      YEAR = {1985},
    NUMBER = {1},
     PAGES = {47--56},
      ISSN = {0166-218X,1872-6771},
   MRCLASS = {05C80 (05C05 60C05)},
  MRNUMBER = {770868},
MRREVIEWER = {Micha\l \ Karo\'{n}ski},
       DOI = {10.1016/0166-218X(85)90058-7},
       URL = {https://doi.org/10.1016/0166-218X(85)90058-7},
}

@article{lueker1981optimization,
  title={Optimization problems on graphs with independent random edge weights},
  author={Lueker, George S},
  journal={SIAM Journal on Computing},
  volume={10},
  number={2},
  pages={338--351},
  year={1981},
  publisher={SIAM}
}

@article {SteelMST,
    AUTHOR = {Steele, J. Michael},
     TITLE = {Growth rates of {E}uclidean minimal spanning trees with power
              weighted edges},
   JOURNAL = {Ann. Probab.},
  FJOURNAL = {The Annals of Probability},
    VOLUME = {16},
      YEAR = {1988},
    NUMBER = {4},
     PAGES = {1767--1787},
      ISSN = {0091-1798,2168-894X},
   MRCLASS = {60F15 (05C05)},
  MRNUMBER = {958215},
MRREVIEWER = {Zdzis\l aw\ Rychlik},
}

@article{prim1957shortest,
  title={Shortest connection networks and some generalizations},
  author={Prim, Robert Clay},
  journal={The Bell System Technical Journal},
  volume={36},
  number={6},
  pages={1389--1401},
  year={1957},
  publisher={Nokia Bell Labs}
}

@article{kruskal1956shortest,
  title={On the shortest spanning subtree of a graph and the traveling salesman problem},
  author={Kruskal, Joseph B},
  journal={Proceedings of the American Mathematical society},
  volume={7},
  number={1},
  pages={48--50},
  year={1956}
}

@article{boruuvka1926jistem,
  title={O jist{\'e}m probl{\'e}mu minim{\'a}ln{\'\i}m},
  author={Bor\r{u}vka, Otakar},
  year={1926}
}

@article{graham1985history,
  title={On the history of the minimum spanning tree problem},
  author={Graham, Ronald L. and Hell, Pavol},
  journal={Annals of the History of Computing},
  volume={7},
  number={1},
  pages={43--57},
  year={1985},
  publisher={IEEE}
}

@article {NPS,
    AUTHOR = {Nair, Chandra and Prabhakar, Balaji and Sharma, Mayank},
     TITLE = {Proofs of the {P}arisi and {C}oppersmith-{S}orkin random
              assignment conjectures},
   JOURNAL = {Random Structures Algorithms},
  FJOURNAL = {Random Structures \& Algorithms},
    VOLUME = {27},
      YEAR = {2005},
    NUMBER = {4},
     PAGES = {413--444},
      ISSN = {1042-9832,1098-2418},
   MRCLASS = {90B36 (05C70 90B50)},
  MRNUMBER = {2178256},
       DOI = {10.1002/rsa.20084},
       URL = {https://doi.org/10.1002/rsa.20084},
}

@article {LinussonWastlund,
    AUTHOR = {Linusson, Svante and W\"{a}stlund, Johan},
     TITLE = {A proof of {P}arisi's conjecture on the random assignment
              problem},
   JOURNAL = {Probab. Theory Related Fields},
  FJOURNAL = {Probability Theory and Related Fields},
    VOLUME = {128},
      YEAR = {2004},
    NUMBER = {3},
     PAGES = {419--440},
      ISSN = {0178-8051,1432-2064},
   MRCLASS = {90C15 (15A52 60C05)},
  MRNUMBER = {2036492},
MRREVIEWER = {Giovanni\ Andreatta},
       DOI = {10.1007/s00440-003-0308-9},
       URL = {https://doi.org/10.1007/s00440-003-0308-9},
}

@article {CaoCLT,
    AUTHOR = {Cao, Sky},
     TITLE = {Central limit theorems for combinatorial optimization problems
              on sparse {E}rd\"os-{R}\'enyi graphs},
   JOURNAL = {Ann. Appl. Probab.},
  FJOURNAL = {The Annals of Applied Probability},
    VOLUME = {31},
      YEAR = {2021},
    NUMBER = {4},
     PAGES = {1687--1723},
      ISSN = {1050-5164,2168-8737},
   MRCLASS = {60F05 (82B44 90B15 90C27 90C35)},
  MRNUMBER = {4312843},
MRREVIEWER = {Tianqing\ Liu},
       DOI = {10.1214/20-aap1630},
       URL = {https://doi.org/10.1214/20-aap1630},
}

@article {MMPCLT,
    AUTHOR = {del Barrio, Eustasio and Loubes, Jean-Michel},
     TITLE = {Central limit theorems for empirical transportation cost in
              general dimension},
   JOURNAL = {Ann. Probab.},
  FJOURNAL = {The Annals of Probability},
    VOLUME = {47},
      YEAR = {2019},
    NUMBER = {2},
     PAGES = {926--951},
      ISSN = {0091-1798,2168-894X},
   MRCLASS = {60F05 (46N30 62E20)},
  MRNUMBER = {3916938},
MRREVIEWER = {Fraser\ Alexander\ Daly},
       DOI = {10.1214/18-AOP1275},
       URL = {https://doi.org/10.1214/18-AOP1275},
}

@book{wastlund2005variance,
  title={The variance and higher moments in the random assignment problem},
  author={W{\"a}stlund, Johan},
  year={2005},
  publisher={Link{\"o}ping University Electronic Press}
}

@article{alm2002exact,
  title={Exact expectations and distributions for the random assignment problem},
  author={Alm, Sven Erick and Sorkin, Gregory B.},
  journal={Combinatorics, Probability and Computing},
  volume={11},
  number={3},
  pages={217--248},
  year={2002},
  publisher={Cambridge University Press}
}

@article {Wastlundcomplete,
    AUTHOR = {W\"{a}stlund, Johan},
     TITLE = {Random matching problems on the complete graph},
   JOURNAL = {Electron. Commun. Probab.},
  FJOURNAL = {Electronic Communications in Probability},
    VOLUME = {13},
      YEAR = {2008},
     PAGES = {258--265},
      ISSN = {1083-589X},
   MRCLASS = {60C05 (90C27)},
  MRNUMBER = {2415133},
MRREVIEWER = {Yuliy\ M.\ Baryshnikov},
       DOI = {10.1214/ECP.v13-1372},
       URL = {https://doi.org/10.1214/ECP.v13-1372},
}

@article{hessler2008concentration,
  title={Concentration of the cost of a random matching problem},
  author={Hessler, Martin and W{\"a}stlund, Johan},
  journal={Preprint. Available at http://www. math. chalmers. se/\~{} wastlund/martingale. pdf},
  year={2008},
  publisher={Citeseer}
}

@article{mezard1986mean,
  title={Mean-field equations for the matching and the travelling salesman problems},
  author={M{\'e}zard, Marc and Parisi, Giorgio},
  journal={Europhysics Letters},
  volume={2},
  number={12},
  pages={913},
  year={1986},
  publisher={IOP Publishing}
}

@incollection{parisi1992spin,
  title={Spin glasses and optimization problems without replicas},
  author={Parisi, Giorgio},
  booktitle={Field Theory, Disorder And Simulations},
  pages={257--284},
  year={1992},
  publisher={World Scientific}
}

@article {AldousZeta,
    AUTHOR = {Aldous, David },
     TITLE = {The {$\zeta(2)$} limit in the random assignment problem},
   JOURNAL = {Random Structures Algorithms},
  FJOURNAL = {Random Structures \& Algorithms},
    VOLUME = {18},
      YEAR = {2001},
    NUMBER = {4},
     PAGES = {381--418},
      ISSN = {1042-9832,1098-2418},
   MRCLASS = {60C05 (60F05)},
  MRNUMBER = {1839499},
MRREVIEWER = {Aart\ J.\ Stam},
       DOI = {10.1002/rsa.1015},
       URL = {https://doi.org/10.1002/rsa.1015},
}

@article {AldousAsymptotics,
    AUTHOR = {Aldous, David},
     TITLE = {Asymptotics in the random assignment problem},
   JOURNAL = {Probab. Theory Related Fields},
  FJOURNAL = {Probability Theory and Related Fields},
    VOLUME = {93},
      YEAR = {1992},
    NUMBER = {4},
     PAGES = {507--534},
      ISSN = {0178-8051,1432-2064},
   MRCLASS = {60C05 (05C70 05C80)},
  MRNUMBER = {1183889},
MRREVIEWER = {Mark\ R.\ Jerrum},
       DOI = {10.1007/BF01192719},
       URL = {https://doi.org/10.1007/BF01192719},
}

@book{papadimitriou1998combinatorial,
  title={Combinatorial optimization: algorithms and complexity},
  author={Papadimitriou, Christos H. and Steiglitz, Kenneth},
  year={1998},
  publisher={Courier Corporation}
}

@article {CoppersmithSorkin,
    AUTHOR = {Coppersmith, Don and Sorkin, Gregory B.},
     TITLE = {Constructive bounds and exact expectations for the random
              assignment problem},
   JOURNAL = {Random Structures Algorithms},
  FJOURNAL = {Random Structures \& Algorithms},
    VOLUME = {15},
      YEAR = {1999},
    NUMBER = {2},
     PAGES = {113--144},
      ISSN = {1042-9832,1098-2418},
   MRCLASS = {05C70 (05C80 68R10 90C15)},
  MRNUMBER = {1704340},
       DOI = {10.1002/(SICI)1098-2418(199909)15:2<113::AID-RSA1>3.0.CO;2-S},
       URL =
              {https://doi.org/10.1002/(SICI)1098-2418(199909)15:2<113::AID-RSA1>3.0.CO;2-S},
}

@article {Lazarus,
    AUTHOR = {Lazarus, Andrew J.},
     TITLE = {Certain expected values in the random assignment problem},
   JOURNAL = {Oper. Res. Lett.},
  FJOURNAL = {Operations Research Letters},
    VOLUME = {14},
      YEAR = {1993},
    NUMBER = {4},
     PAGES = {207--214},
      ISSN = {0167-6377,1872-7468},
   MRCLASS = {90C27 (90C15)},
  MRNUMBER = {1250142},
       DOI = {10.1016/0167-6377(93)90071-N},
       URL = {https://doi.org/10.1016/0167-6377(93)90071-N},
}

@incollection {KarpUppbnd,
    AUTHOR = {Karp, Richard M.},
     TITLE = {An upper bound on the expected cost of an optimal assignment},
 BOOKTITLE = {Discrete algorithms and complexity ({K}yoto, 1986)},
    SERIES = {Perspect. Comput.},
    VOLUME = {15},
     PAGES = {1--4},
 PUBLISHER = {Academic Press, Boston, MA},
      YEAR = {1987},
      ISBN = {0-12-386870-X},
   MRCLASS = {90C10},
  MRNUMBER = {910922},
       DOI = {10.2167/pst001.0},
       URL = {https://doi.org/10.2167/pst001.0},
}

@article {linearprograms,
    AUTHOR = {Dyer, M. E. and Frieze, A. M. and McDiarmid, C. J. H.},
     TITLE = {On linear programs with random costs},
   JOURNAL = {Math. Programming},
  FJOURNAL = {Mathematical Programming},
    VOLUME = {35},
      YEAR = {1986},
    NUMBER = {1},
     PAGES = {3--16},
      ISSN = {0025-5610,1436-4646},
   MRCLASS = {90C05},
  MRNUMBER = {842630},
MRREVIEWER = {O.\ Pokorn\'{a}},
       DOI = {10.1007/BF01589437},
       URL = {https://doi.org/10.1007/BF01589437},
}

@article {Walkup79,
    AUTHOR = {Walkup, David W.},
     TITLE = {On the expected value of a random assignment problem},
   JOURNAL = {SIAM J. Comput.},
  FJOURNAL = {SIAM Journal on Computing},
    VOLUME = {8},
      YEAR = {1979},
    NUMBER = {3},
     PAGES = {440--442},
      ISSN = {0097-5397},
   MRCLASS = {68E99 (05C99)},
  MRNUMBER = {539262},
       DOI = {10.1137/0208036},
       URL = {https://doi.org/10.1137/0208036},
}

@article{kurtzberg1962approximation,
  title={On approximation methods for the assignment problem},
  author={Kurtzberg, Jerome M},
  journal={Journal of the ACM (JACM)},
  volume={9},
  number={4},
  pages={419--439},
  year={1962},
  publisher={ACM New York, NY, USA}
}

@article{munkres1957algorithms,
  title={Algorithms for the assignment and transportation problems},
  author={Munkres, James},
  journal={Journal of the society for industrial and applied mathematics},
  volume={5},
  number={1},
  pages={32--38},
  year={1957},
  publisher={SIAM}
}

@article{konig1916graphen,
  title={{\"U}ber graphen und ihre anwendung auf determinantentheorie und mengenlehre},
  author={K{\"o}nig, D{\'e}nes},
  journal={Mathematische Annalen},
  volume={77},
  number={4},
  pages={453--465},
  year={1916},
  publisher={Springer}
}

@article{krauth1989cavity,
  title={The cavity method and the travelling-salesman problem},
  author={Krauth, Werner and M{\'e}zard, Marc},
  journal={Europhysics Letters},
  volume={8},
  number={3},
  pages={213},
  year={1989},
  publisher={IOP Publishing}
}

@article {WastlundReplicaRA,
    AUTHOR = {W\"{a}stlund, Johan},
     TITLE = {Replica symmetry of the minimum matching},
   JOURNAL = {Ann. of Math. (2)},
  FJOURNAL = {Annals of Mathematics. Second Series},
    VOLUME = {175},
      YEAR = {2012},
    NUMBER = {3},
     PAGES = {1061--1091},
      ISSN = {0003-486X,1939-8980},
       DOI = {10.4007/annals.2012.175.3.2},
       URL = {https://doi.org/10.4007/annals.2012.175.3.2},
}

@article{sourlas1986statistical,
  title={Statistical mechanics and the travelling salesman problem},
  author={Sourlas, Nicolas},
  journal={Europhysics Letters},
  volume={2},
  number={12},
  pages={919},
  year={1986},
  publisher={IOP Publishing}
}

@phdthesis{percus1997voyageur,
  title={Voyageur de commerce et problemes stochastiques associ{\'e}s},
  author={Percus, Allon G.},
  year={1997},
  school={Paris 6}
}

@article{parisi2001neighborhood,
  title={Neighborhood preferences in random matching problems},
  author={Parisi, G and Rati{\'e}ville, M},
  journal={The European Physical Journal B-Condensed Matter and Complex Systems},
  volume={22},
  pages={229--237},
  year={2001},
  publisher={Springer}
}

@article{pagnani2003near,
  title={Near-optimal configurations in mean-field disordered systems},
  author={Pagnani, Andrea and Parisi, Giorgio and Rati{\'e}ville, Mathieu},
  journal={Physical Review E},
  volume={68},
  number={4},
  pages={046706},
  year={2003},
  publisher={APS}
}

@article{mezard1985replicas,
  title={Replicas and optimization},
  author={M{\'e}zard, Marc and Parisi, Giorgio},
  journal={Journal de Physique Lettres},
  volume={46},
  number={17},
  pages={771--778},
  year={1985},
  publisher={Les Editions de Physique}
}

@book{mezard1987spin,
  title={Spin glass theory and beyond: An Introduction to the Replica Method and Its Applications},
  author={M{\'e}zard, Marc and Parisi, Giorgio and Virasoro, Miguel Angel},
  volume={9},
  year={1987},
  publisher={World Scientific Publishing Company}
}

@article{percus1999stochastic,
  title={The stochastic traveling salesman problem: Finite size scaling and the cavity prediction},
  author={Percus, Allon G. and Martin, Olivier C.},
  journal={Journal of Statistical Physics},
  volume={94},
  pages={739--758},
  year={1999},
  publisher={Springer}
}

@article {Houdayer,
    AUTHOR = {Houdayer, J. and Boutet de Monvel, J. H. and Martin, O. C.},
     TITLE = {Comparing mean field and {E}uclidean matching problems},
   JOURNAL = {Eur. Phys. J. B Condens. Matter Phys.},
  FJOURNAL = {The European Physical Journal. B. Condensed Matter Physics},
    VOLUME = {6},
      YEAR = {1998},
    NUMBER = {3},
     PAGES = {383--393},
      ISSN = {1434-6028,1434-6036},
   MRCLASS = {82B44 (82D30)},
  MRNUMBER = {1665617},
       DOI = {10.1007/s100510050565},
       URL = {https://doi.org/10.1007/s100510050565},
}

@article{cerf1997random,
  title={The random link approximation for the Euclidean traveling salesman problem},
  author={Cerf, Nicolas J. and De Monvel, J. BOUTET and Bohigas, Oriol and Martin, Olivier C. and Percus, Allon G.},
  journal={Journal de Physique I},
  volume={7},
  number={1},
  pages={117--136},
  year={1997},
  publisher={EDP Sciences}
}

@article{brunetti1991extensive,
  title={Extensive numerical simulations of weighted matchings: Total length and distribution of links in the optimal solution},
  author={Brunetti, Romeo and Krauth, W and M{\'e}zard, M and Parisi, G},
  journal={Europhysics Letters},
  volume={14},
  number={4},
  pages={295},
  year={1991},
  publisher={IOP Publishing}
}

@article{boettcher2000nature,
  title={Nature's way of optimizing},
  author={Boettcher, Stefan and Percus, Allon G.},
  journal={Artificial Intelligence},
  volume={119},
  number={1-2},
  pages={275--286},
  year={2000},
  publisher={Elsevier}
}

@article {Bacci,
    AUTHOR = {Bacci, S. and Miranda, E. N.},
     TITLE = {The traveling salesman problem and its analogy with
              two-dimensional spin glasses},
   JOURNAL = {J. Statist. Phys.},
  FJOURNAL = {Journal of Statistical Physics},
    VOLUME = {56},
      YEAR = {1989},
    NUMBER = {3-4},
     PAGES = {547--551},
      ISSN = {0022-4715,1572-9613},
   MRCLASS = {90B10 (82A57 90C27)},
  MRNUMBER = {1009516},
       DOI = {10.1007/BF01044452},
       URL = {https://doi.org/10.1007/BF01044452},
}

@misc{kirkpatrick1981lecture,
  title={Lecture notes in physics},
  author={Kirkpatrick, S.},
  year={1981},
  publisher={Springer Berlin}
}

@article{donath1969algorithm,
  title={Algorithm and average-value bounds for assignment problems},
  author={Donath, WE},
  journal={IBM Journal of Research and Development},
  volume={13},
  number={4},
  pages={380--386},
  year={1969},
  publisher={IBM}
}

@book {YukichBook,
    AUTHOR = {Yukich, Joseph E.},
     TITLE = {Probability theory of classical {E}uclidean optimization
              problems},
    SERIES = {Lecture Notes in Mathematics},
    VOLUME = {1675},
 PUBLISHER = {Springer-Verlag, Berlin},
      YEAR = {1998},
     PAGES = {x+152},
      ISBN = {3-540-63666-8},
   MRCLASS = {60C05 (60D05 60F10 60F15 60G55 90C27)},
  MRNUMBER = {1632875},
MRREVIEWER = {Mathew\ D.\ Penrose},
       DOI = {10.1007/BFb0093472},
       URL = {https://doi.org/10.1007/BFb0093472},
}

@book {SteeleBook,
    AUTHOR = {Steele, J. Michael},
     TITLE = {Probability theory and combinatorial optimization},
    SERIES = {CBMS-NSF Regional Conference Series in Applied Mathematics},
    VOLUME = {69},
 PUBLISHER = {Society for Industrial and Applied Mathematics (SIAM),
              Philadelphia, PA},
      YEAR = {1997},
     PAGES = {viii+159},
      ISBN = {0-89871-380-3},
   MRCLASS = {60-02 (60C05 60D05 60E15 60F05 90C27)},
  MRNUMBER = {1422018},
       DOI = {10.1137/1.9781611970029},
       URL = {https://doi.org/10.1137/1.9781611970029},
}

@book{fvbook,
  title={Statistical mechanics of lattice systems: a concrete mathematical introduction},
  author={Friedli, Sacha and Velenik, Yvan},
  year={2017},
  publisher={Cambridge University Press}
}

@article {DZ16,
    AUTHOR = {Dey, Partha S. and Zygouras, Nikos},
     TITLE = {High temperature limits for {$(1+1)$}-dimensional directed
              polymer with heavy-tailed disorder},
   JOURNAL = {Ann. Probab.},
  FJOURNAL = {The Annals of Probability},
    VOLUME = {44},
      YEAR = {2016},
    NUMBER = {6},
     PAGES = {4006--4048}
}

@article {CometsYoshida,
    AUTHOR = {Comets, Francis and Yoshida, Nobuo},
     TITLE = {Directed polymers in random environment are diffusive at weak
              disorder},
   JOURNAL = {Ann. Probab.},
  FJOURNAL = {The Annals of Probability},
    VOLUME = {34},
      YEAR = {2006},
    NUMBER = {5},
     PAGES = {1746--1770},
      ISSN = {0091-1798,2168-894X},
   MRCLASS = {60K37 (60J60 60K35 82C44 82D60)},
  MRNUMBER = {2271480},
MRREVIEWER = {Marina\ Vachkovskaia},
       DOI = {10.1214/009117905000000828},
       URL = {https://doi-org.libproxy.lib.unc.edu/10.1214/009117905000000828},
}

@misc{Dey_Poi_notes,
  title={Stein--Chen Method for Poisson Approximation},
  author={Dey, Partha S.~},
  year={2014},
    note={Available at \href{https://psdey.web.illinois.edu/414CourseNotes.pdf}{https://psdey.web.illinois.edu/414CourseNotes.pdf}}
}

@article{AGG90,
author = {Richard Arratia and Larry Goldstein and Louis Gordon},
title = {{Poisson Approximation and the Chen-Stein Method}},
volume = {5},
journal = {Statistical Science},
number = {4},
publisher = {Institute of Mathematical Statistics},
pages = {403 -- 424},
year = {1990},
doi = {10.1214/ss/1177012015},
URL = {https://doi.org/10.1214/ss/1177012015}
}
